%% file: paper4-6.tex
\documentclass[a4paper,11pt,reqno]{amsart}
\pdfoutput=1

\usepackage[top=30truemm,bottom=30truemm,left=30truemm,right=30truemm]{geometry}

\input{mizuno_preamble}

\usepackage[backend=bibtex, style=alphabetic, sorting=nty, giveninits=true, backref=false, backrefstyle=none, doi=false]{biblatex} 

\bibliography{mizuno_paper}
\AtBeginBibliography{\footnotesize}
\DeclareFieldFormat{biblabeldate}{\mkbibparens{\mkbibbold{#1}}}\DeclareFieldFormat[article,periodical]{volume}{\mkbibbold{#1}}\DeclareFieldFormat[article,periodical]{number}{\mkbibbold{#1}}\renewbibmacro{in:}{}

\title[NC Bondal-Orlov's reconstruction theorem]{Bondal-Orlov's reconstruction theorem in noncommutative projective geometry}
 
\author{Yuki Mizuno}
\email{mizuno.y@aoni.waseda.jp, m7d5932a72xxgxo@fuji.waseda.jp}
\date{}
\address{Department~of~Mathematics, School~of~Science~and~Engineering, Waseda~University, Ohkubo~3-4-1, Shinjuku, Tokyo~169-8555, Japan}

\keywords{Noncommutative algebraic geometry, Noncommutative projective geometry, Derived category, Dg category}
\subjclass[2020]{14A22, 14F08}

\begin{document}

\begin{abstract}
    We show that Bondal-Orlov's reconstruction theorem holds in noncommutative projective geometry. 
    We also prove that fully faithful exact functors between derived categories of noncommutative projective schemes are of Fourier-Mukai type.
\end{abstract}

\maketitle

\section{Introduction}
Whether a scheme is determined by its (derived) category of (quasi-)coherent sheaves is a fundamental problem in algebraic geometry.
The first result in this direction is obtained by Gabriel.
\begin{thm}[{\cite{gabriel1962categories}}]
    Let $X, Y$ be noetherian schemes.
    \[
    \Coh(X) \simeq \Coh(Y) \Rightarrow X \simeq Y.
    \]
\end{thm}
This result means that the category of coherent sheaves on a noetherian scheme determines the scheme.
Generalizations of Gabriel's theorem to more general settings have been studied by many authors (\cite{rosenberg1998spectrum}, \cite{brandenberg2018rosenberg}, \cite{antieau2016reconstruction}, \cite{calabrase2015moduli} and \cite{perego2009gabriel}).

In contrast to the case of the category of coherent sheaves, the bounded derived category of coherent sheaves on a noetherian scheme does not determine the scheme.
Mukai first discovered this fact (\cite{mukai1981duality}).
He proved that the bounded derived category of coherent sheaves on an abelian variety is equivalent to that on the dual abelian variety.
On the other hand, Bondal and Orlov showed that smooth projective varieties with (anti-)ample canonical bundles are determined by their bounded derived categories of coherent sheaves.

\begin{thm}[{\cite[Theorem 2.5]{Bondal2001reconstruction}}]
\label{thm:BO-reconstruction}
    Let $X, Y$ be smooth projective varieties over a field $k$.
    If the canonical bundles $\canmod_X, \canmod_Y$ are (anti-)ample, then
    \[
    D^b(\Coh(X)) \simeq D^b(\Coh(Y)) \Rightarrow X \simeq Y.
    \]
    
\end{thm}

\begin{rmk}
    In \cref{thm:BO-reconstruction}, actually, it is sufficient to assume that either $\canmod_X$ or $\canmod_Y$ is (anti-)ample.
\end{rmk}

This reconstruction theorem has been generalized to some other settings (\cite{ballard2011derived}, \cite{calabrese2018relative} and \cite{sancho2012reconstructing}).
The method of the proof of \cref{thm:BO-reconstruction} has also applications such as determining the groups of autoequivalences of derived categories of coherent sheaves.

On the other hand, Artin and Zhang introduced the notion of noncommutative projective schemes (\cite{artin1994noncommutative}) from the viewpoint of the following result by Serre.
\begin{thm}[{\cite{serre1955faisceaux}}]
    Let $R$ be a commutative finitely generated graded $k$-algebra which is also generated by $R_1$ as 
    $R_0$-algebra.
    Then, 
    \[
    \Coh(\Proj(R)) \simeq \qgr (R),
    \]
    where $\qgr R$ is the Serre quotient category of the category $\gr R$ of finitely generated graded $R$-modules by the category $\tor R$ of finitely generated torsion graded $R$-modules.
\end{thm}
For any (not necessarily commutative) right noetherian $\Nbb$-graded $k$-algebra $R$, the noncommutative projective scheme associated to $R$ is defined as the quotient category $\qgr R$ of the category $\gr R$ of finitely generated graded right $R$-modules by the category $\tor R$ of finitely generated torsion graded right $R$-modules.
Considering Gabriel's result and others, the notion of a noncommutative projective scheme can be considered to be a generalization of the notion of a commutative projective scheme.
Furthermore, in noncommutative projective geometry, interesting objects such as Artin-Schelter (AS) regular algebras have been discovered and actively studied by many researchers.

In this paper, we consider the noncommutative version of \cref{thm:BO-reconstruction}.

Let $k$ be a field.
Let $A, B$ be noetherian (i.e., left and right noetherian) locally finite  $\Nbb$-graded $k$-algebras.
We assume that $A,B$ have balanced dualizing complexes (see \cref{def:dualizing complex}).
A canonical bimodule $\canmod_A$ for $\qgr (A)$ is a pair $(- \otimes \canmod_A, \Hom(\canmod_A,-))$ of autoequivalences of $\qgr (A)$ such that for some $n \in \Nbb$, the induced autoequivalence $- \otimes \canmod_A [n]$ of $D^b(\qgr (A))$ gives a Serre functor for $D^b(\qgr (A))$ (see \cref{def:bimodules}).
Let $\Ocal_A(i)$ be the object in $\QGr (A)$ corresponding to the right graded $A$-module $A(i)$.
The main theorem in this paper is the following.

\begin{thm}[= \cref{thm:re-main1}]
\label{thm:main1}
   We assume that $\qgr (A), \qgr (B)$ have canonical bimodules $\canmod_A, \canmod_B$, respectively.

    If $ - \otimes \canmod_A, - \otimes \canmod_B$ are AZ-(anti-)ample (see \cref{def:AZ-ample}), then 
    \[
     D^b(\qgr (A)) \simeq D^b(\qgr (B)) \Rightarrow \qgr(A) \simeq \qgr(B).
    \]

\end{thm}
AZ-(anti-)ampleness is a generalization of (anti-)ampleness of line bundles on a projective scheme in noncommutative projective geometry, which is introduced in \cite{artin1994noncommutative}.
If we assume that $A,B$ are commutative connected graded $k$-algebras generated in degree 1, then \cref{thm:main1} recovers \cref{thm:BO-reconstruction} since any projective scheme over $k$ is isomorphic to a projective scheme associated to a commutative connected finitely generated graded $k$-algebra generated in degree 1.
The result also has an application in the study of AS regular algebras.

\begin{cor}[= \cref{cor:re-main1}]
\label{cor:main1}
Let $A,B$ be noetherian AS-regular algebras over $A_0, B_0$, respectively (see \cref{def:AS-regular-Gorenstein}).
Then, 
\[
     D^b(\qgr (A)) \simeq D^b(\qgr (B)) \Rightarrow \qgr (A) \simeq \qgr (B).
    \]
\end{cor}
To prove the corollary, the notion of quasi-Veronese algebras introduced in \cite{mori2013bconstruction} is useful.
This corollary follows from the observation that a noncommutative projective scheme associated to a quasi-Veronese algebra of an AS-regular algebra is isomorphic to the original noncommutative projective scheme and the fact that the canonical bimodule of the noncommutative projective scheme of an appropriate quasi-Veronese algebra of an AS regular algebra is AZ-anti-ample.
In particular, even when proving the corollary for connected graded AS-regular algebras, it is necessary to treat locally finite AS-regular algebras.
This is the reason why we prove \cref{thm:main1} for noncommutative projective schemes associated to locally finite graded $k$-algebras.

\bigbreak
We also study Fourier-Mukai functors between derived categories of noncommutative projective schemes.
Let $F:\Perf(\QGr(A)) \rightarrow D(\QGr (B))$ be an exact functor.
$F$ is called of \emph{Fourier-Mukai type} if there exists an object $\qgrobjname{E} \in D(\QBiGr(A^{\op} \otimes B))$ such that we have an isomorphism of functors $F(-) \simeq \Phi_\qgrobjname{E}(-):= \pi_B(\Romega_A(-)\otimes_{\Acal}^{\LD} \Romega_{A^{\op}\otimes_k B}(\qgrobjname{E}))$,
 where $\Acal$ is the associated dg category to $A$ and $\otimes_{\Acal}$ means the tensor product of dg $\Acal$ -modules.
 The objects of $\dgcatname{A}$ are the integers and the morphism spaces between $i,j \ (i,j \in \Zbb)$ are $\Hom_{\dgcatname{A}}(i,j)=A^{j-i}$.
 Note also that graded right (resp. left) $A$-modules can be considered as right (resp. left) dg $\Acal$-modules (in detail, see \cref{sec:FM in ncproj}).

We show that exact fully faithful functors between derived categories of noncommutative projective schemes are of Fourier-Mukai type.
In commutative algebraic geometry, whether exact functors between derived categories of coherent sheaves are of Fourier-Mukai type is studied in many settings (\cite{orlov1997equivalences}, \cite{lunts2010uniqueness}, \cite{cananoco2014fourier}, \cite{ballard2009equivalences} etc., see also \cite{cananoco2012fourier}). 
The second main theorem in this paper is the following.

\begin{thm}[= \cref{thm:re-main2}]
\label{thm:main2}
Let $F: \Perf(\QGr (A)) \rightarrow D(\QGr (B))$ be an exact fully faithful functor.
Then, there exists an object $\qgrobjname{E} \in D(\QBiGr(A^{\mathrm{op}} \otimes_k B))$ such that 
\begin{enumerate}
    \item $\Phi_{\qgrobjname{E}}$ is exact fully faithful and $\Phi_\qgrobjname{E}(\qgrobjname{P}) \simeq F(\qgrobjname{P})$ for any $\qgrobjname{P} \in \Perf(\QGr (A))$.
    \item If $F$ sends a perfect complex to a perfect complex, then the induced functor $\Phi_\qgrobjname{E}:D(\QGr (A)) \rightarrow D(\QGr (B))$ is fully faithful and $\Phi_\qgrobjname{E}$ sends a perfect complex to a perfect complex.
    \item If $\R^1\torfunct{A}(A)$ is a finite $A$-module, then $F \simeq \Phi_\qgrobjname{E}$.
    Moreover, if $\Phi_{\qgrobjname{E}} \simeq \Phi_{\qgrobjname{E}'}$ for some $\qgrobjname{E}' \in D(\QBiGr(A^{\mathrm{op}} \otimes_k B))$, then $\qgrobjname{E} \simeq \qgrobjname{E}'$.
\end{enumerate}
\end{thm} 
This theorem is a generalization of \cite[Corollary 9.13]{lunts2010uniqueness} to the setting of noncommutative projective geometry.
In the same way, we can obtain a simpler version of \cref{thm:main2}.

\begin{prop}[= \cref{prop:re-main2}]
\label{prop:main2}
Let $\Perf(\QGr (A)) \rightarrow D(\QGr (B))$ be an exact fully faithful functor.
We assume that $H^0(\QGr (A),\Ocal_A(m))=0,\ m \ll 0$.
Then, $F$ is of Fourier-Mukai type and the Fourier-Mukai kernels are unique up to quasi-isomorphism.
 \end{prop}

To prove \cref{thm:main2}, we use methods in \cite{lunts2010uniqueness}.
In particular, we have to construct an ample sequence of objects in $\Perf(\QGr (A))$ in the sense of Bondal-Orlov (\cite{Bondal2001reconstruction}, \cite{lunts2010uniqueness}).
In this paper, ample sequences are referred to as BO-ample sequences to distinguish the notion of BO-ampleness from that of AZ-ampleness.
However, because we cannot use geometric techniques as in the case of commutative algebraic geometry, we need to use some other algebraic methods to construct a BO-ample sequence in $\Perf(\QGr (A))$.
The notion of quasi-Veronese algebras is again useful.
Taking quasi-Veronese algebras means that we take another polarization of the noncommutative projective scheme.
In particular, considering quasi-Veronese algebra is convenient because higher quasi-Veronese algebras are generated in degree 1.
These discussions including the proof of  \cref{thm:main2} are in \cref{sec:FM in ncproj}.

To prove \cref{thm:main1}, we use techniques in \cite{ballard2021kernels}.
In \cite{ballard2021kernels}, the authors extended results in \cite{toen2007homotopy} to the setting of noncommutative projective geometry.
In \cref{sec:BO for ncproj}, we study properties of Fourier-Mukai functors in noncommutative projective geometry. 
In the original proof of \cref{thm:BO-reconstruction}, the notions of point-like objects and invertible objects by Bondal and Orlov.
However, in noncommutative projective geometry, it is difficult to study these objects.
For example, this is because simple objects in $\qgr (A)$ are more complicated than simple objects in $\Coh(X)$.
Therefore, we use ideas of arguments in \cite[Section 6]{huybrechts2006fourier}, where the main technique is the use of Fourier-Mukai functors.
By proving generalized results in \cite{huybrechts2006fourier}, we can prove \cref{thm:main1}.

\subsection*{Acknowledgements}
The author would like to thank Yuki Imamura for helpful discussions and comments.
This work was supported by JSPS KAKENHI Grant Number 24K22841.

\section{Preliminaries}
\label{sec:prelim}

In this paper, we work over a field $k$.
\subsection{Dg-categories}
\label{subsec:dgcat}
We recall basic notions about dg-categories and quasi-functors.
We refer the reader to \cite{keller2006differential} and \cite{chen2021informal} for good surveys.

A \emph{dg category} is a $k$-linear category such that, for all $A,B \in \Ob(\dgcatname{A})$, the morphism spaces $\Hom_\dgcatname{A}(A,B)$ are $\Zbb$-graded  $k$-modules with a differential $d:\Hom_\dgcatname{A}(A,B) \rightarrow \Hom_\dgcatname{A}(A,B)$ of degree 1 and the composition maps are morphisms of complexes.
A \emph{dg functor} $\dgfunname{F}:\dgcatname{A} \rightarrow \dgcatname{B}$ between dg categories is given by a map $\Ob(A) \rightarrow \Ob(B)$ and morphisms of complexes of $k$-modules $\Hom_\dgcatname{A}(A,B) \rightarrow \Hom_\dgcatname{B}(F(A),F(B))$ compatible with the compositions and the units.
The category with objects dg categories and
morphisms dg functors is denoted by $\dgcat$.
Given two dg categories $\dgcatname{A}, \dgcatname{B}$, one can construct the dg categories $\Fundg(\dgcatname{A}, \dgcatname{B})$ and $\dgcatname{A} \otimes \dgcatname{B}$ (for detail, see \cite[Section 2.3]{keller2006differential}).
The objects of $\Fundg(\dgcatname{A}, \dgcatname{B})$ are dg functors from $\dgcatname{A}$ to $\dgcatname{B}$ and the morphisms are dg natural transformations of dg functors from $\dgcatname{A}$ to $\dgcatname{B}$.
The objects of $\dgcatname{A} \otimes \dgcatname{B}$ are pairs $(A,B)$ of objects $A \in \dgcatname{A}$ and $B \in \dgcatname{B}$ and the morphisms are given by 
\[
    \Hom_{\dgcatname{A} \otimes \dgcatname{B}}((A,B),(A',B')) = \Hom_\dgcatname{A}(A,A') \otimes_k \Hom_\dgcatname{B}(B,B').
\]
For dg categories $\dgcatname{A}, \dgcatname{B}, \dgcatname{C}$, we have the natural isomorphism
\[
    \Fundg(\dgcatname{A} \otimes \dgcatname{B}, \dgcatname{C}) \simeq \Fundg(\dgcatname{A}, \Fundg(\dgcatname{B}, \dgcatname{C})).
\] 

For a dg categories $\dgcatname{A}$, we denote by $Z^0(\dgcatname{A})$ and $H^0(\dgcatname{A})$ the \emph{standard category} and the \emph{homotopy category} of $\dgcatname{A}$, respectively.
$Z^0(\dgcatname{A})$ and $H^0(\dgcatname{A})$ are the categories with the same objects as those of $\dgcatname{A}$ and morphisms 
\[
    \Hom_{Z^0(\dgcatname{A})}(A,B) = Z^0(\Hom_\dgcatname{A}(A,B)), \quad \Hom_{H^0(\dgcatname{A})}(A,B) = H^0(\Hom_\dgcatname{A}(A,B)).
\]
For a dg-category $\dgcatname{A}$, $\dgcatname{A}^{\op}$ denotes the \emph{opposite category} of $\dgcatname{A}$.
Its objects are the same as those of $\dgcatname{A}$ and for all $A,B \in \Ob(\dgcatname{A})$, the morphisms are defined by 
\[
\Hom_{\dgcatname{A}^{\op}}(A,B) = \Hom_{\dgcatname{A}}(B,A).
\]
The composition of $f \in \Hom_{\dgcatname{A}^{\op}}(A,B)$ and $g \in \Hom_{\dgcatname{A}^{\op}}(B,C)$ is defined by $(-1)^{\deg(f)\deg(g)}fg$.

For a dg category $\dgcatname{A}$, we define a \emph{right} \emph{dg $\dgcatname{A} $-module} as a dg-functor $\dgcatname{A}^{\op} \rightarrow \Cdg(k)$, where $\Cdg(k)$ is the dg category of complexes of $k$-modules.
A \emph{left dg $\dgcatname{A}$-module} is a dg functor $\dgcatname{A} \rightarrow \Cdg(k)$.
Let $\dgcatname{B}$ be another dg category.
A \emph{dg $\dgcatname{A}$-$\dgcatname{B}$-bimodule} is a dg functor $\dgcatname{A} \otimes \dgcatname{B}^{\op} \rightarrow \Cdg(k)$.
We set $\dgmod(\dgcatname{A}) = \Fundg(\dgcatname{A}^{\op}, \Cdg(k))$.
When we simply say a dg $\dgcatname{A}$-module, we mean a right dg $\dgcatname{A}$-module.
We denote by $\Acydg(\dgcatname{A})$ the full dg subcategory of $\dgmod(\dgcatname{A})$ consisting of acyclic dg modules.
The \emph{dg derived category} $D_{\dg}(\dgcatname{A})$ is the dg quotient $\dgmod(\dgcatname{A})/\Acydg(\dgcatname{A})$.
The \emph{derived category} $D(\dgcatname{A})$ is the Verdier quotient $H^0(\dgmod(\dgcatname{A}))/H^0(\Acydg(\dgcatname{A}))$.
For each object $X \in \dgcatname{A}$, we have the right dg module $\yoneda_\dgcatname{A}(X)$ \emph{representable} by X, which induces the dg Yoneda functor 
\[
\yoneda_\dgcatname{A}: \dgcatname{A} \rightarrow \dgmod(\dgcatname{A}), \quad X \mapsto \yoneda_\dgcatname{A}(X):=\dgcatname{A}(-,X).
\]
We denote by $\bar{\dgcatname{A}} \subset \hproj(\dgcatname{A})$ the full dg subcategory with objects the dg modules which are homotopy equivalent to objects in the image of $\yoneda_\dgcatname{A}$.

A dg-module is \emph{free} if it is isomorphic to a direct sum of dg-modules of the form $\yoneda_\dgcatname{A}(X)[n]$, where $X \in \dgcatname{A}$ and $n \in \Zbb$.
A dg-module $\dgmodname{M}$ is \emph{semi-free} if it has a filtration 
\[
0=\dgmodname{M}_0 \subset \dgmodname{M}_1 \subset \cdots\subset \dgmodname{M}_i \subset \cdots \subset  \dgmodname{M}
\] such that $\dgmodname{M}_i/\dgmodname{M}_{i-1}$ is free for all $i$.
The full dg subcategory of semi-free dg modules is denoted by $\sfmod(\dgcatname{A})$.
A semi-free dg module $\dgmodname{M}$ is \emph{finitely generated} if $\dgmodname{M}_n=\dgmodname{M}$ for some $n \in \Nbb$ and $\dgmodname{M}_i/\dgmodname{M}_{i-1}$ is a finite direct sum of dg modules of the form $\yoneda_\dgcatname{A}(X)[n]$.
A dg module $\dgmodname{M}$ is \emph{perfect} if it is homotopy equivalent to a direct summand of a finitely generated semi-free dg module.
We denote by $\Perfdg (\dgcatname{A})$ the full dg subcategory of $\dgmod(\dgcatname{A})$ consisting of perfect dg modules.  
For a dg-module $\dgmodname{P}$, $\dgmodname{P}$ is called \emph{h-projective} if 
\[
H^0(\Hom_{\dgmod(\dgcatname{A})}(\dgmodname{P},\dgmodname{N}))=0
\] for every acyclic dg module $\dgmodname{N}$.
We denote by $\hproj(\dgcatname{A})$ the full dg subcategory of $\dgmod(\dgcatname{A})$ consisting of h-projective dg modules.
Note that we have the canonical dg functors 
\[
\sfmod(\dgcatname{A}) \hookrightarrow \hproj(\dgcatname{A}) \hookrightarrow \dgmod(\dgcatname{A}),
\]
which induce the equivalences
\[
H^0(\sfmod(\dgcatname{A})) \simeq H^0(\hproj(\dgcatname{A})) \simeq D(\dgcatname{A}).
\]

\begin{dfn}
    \label{def:compact object}
    Let $\Tcal$ be a triangulated category that admits arbitrary direct sums.
    An object $X \in \Tcal$ is called \emph{compact} if the functor $\Hom_\Tcal(X,-)$ commutes with arbitrary direct sums.
    We denote by $\Tcal^c$ the full subcategory of compact objects in $\Tcal$.
    Let $S \subset \Ob(\Tcal^c)$ is called a set of compact generators if any object $Y$ in $\Tcal$ such that $\Hom_\Tcal(X,Y[n])=0$ for all $X \in S$ and all $n \in \Zbb$ is the zero object.
\end{dfn}

\begin{eg}[See also {\cite[Example 1.9]{lunts2010uniqueness}}]
    \label{eg:perfect dg module}
Let $\dgcatname{A}$ be a dg-category.
The set $\{\yoneda_\dgcatname{A}(X)\}_{X \in \dgcatname{A}}$ is a set of compact generators for $D(\dgcatname{A})$.
The subcategory $D(\dgcatname{A})^c$ of compact objects coincides with the subcategory  $\Perf(\dgcatname{A})$ of perfect dg modules.
\end{eg}

Let $\dgmodname{M} \in \dgmod(\dgcatname{A})$ and $\dgmodname{N} \in \dgmod(\dgcatname{A}^{\op})$.
Then, the tensor product $\dgmodname{M} \otimes_\dgcatname{A} \dgmodname{N}$ is defined as 
\begin{align*}
    \dgmodname{M} \otimes_\dgcatname{A} \dgmodname{N}:= \Cok \left(\Xi: \bigoplus_{A,B \in \dgcatname{A}} \dgmodname{M}(A) \otimes_k \Hom_{\dgcatname{A}}(B, A) \otimes_k \dgmodname{N}(B) \rightarrow \bigoplus_{C \in \dgcatname{A}} \dgmodname{M}(C) \otimes_k \dgmodname{N}(C)  \right),
\end{align*}
where 
\begin{align*}
    \Xi((m,f,n)):=&\dgmodname{M}(f)(m) \otimes n -(-1)^{\deg(m)\deg(f)}
     m \otimes \dgmodname{N}(f)(n) \\
    &\in (\dgmodname{M}(A) \otimes_k \dgmodname{N}(A)) \oplus (\dgmodname{M}(B) \otimes_k \dgmodname{N}(B)).
\end{align*}

In addition, let $\dgmodname{M} \in \dgmod(\dgcatname{A^{\op} \otimes B})$ and $\dgmodname{N} \in \dgmod(\dgcatname{B^{\op} \otimes C})$.
Then, the tensor product $\dgmodname{M} \otimes_{\dgcatname{B}} \dgmodname{N} \in \dgmod(\dgcatname{A^{\op} \otimes C})$ is defined as
\[
(\dgmodname{M} \otimes_{\dgcatname{B}} \dgmodname{N})(A,C):= \dgmodname{M}(A,-) \otimes_{\dgcatname{B}} \dgmodname{N}(-,C)
\]
for any $A \in \dgcatname{A}$ and $C \in \dgcatname{C}$.

\begin{dfn}[See also {\cite[Definition 3.2, Propostion 3.6]{genovese2017adjunctions}}]
    \label{def:coend}
    Let $\dgfunname{F}: \dgcatname{A}^{\op} \otimes \dgcatname{A} \rightarrow \Cdg(k)$ be a dg $\dgcatname{A}$-$\dgcatname{A}$-bimodule.
    The \emph{coend} $\int^{A \in \dgcatname{A}} F(A,A)$ of $F$ is defined by 

    \begin{align*}
       \int^{A \in \dgcatname{A}} F(A,A) := \Cok \left(\bigoplus_{A,B \in \dgcatname{A}} \dgcatname{A}(B,A) \otimes F(A,B) \rightarrow \bigoplus_{C \in \dgcatname{A}} F(C,C) \right), \\
        f \otimes x \mapsto f \cdot x - (-1)^{\deg(f)\deg(x)} x \cdot f.
    \end{align*}
\end{dfn}

\begin{eg}[See also {\cite[Appendix C.3]{drinfeld2004dgquotients}, \cite[Example 2.13]{imamura2024formal}}]
Let $\dgmodname{M},\dgmodname{N}$ be a right and left dg $\dgcatname{A}$-module, respectively.
Then, 
\[
\dgmodname{M} \otimes_\dgcatname{A} \dgmodname{N} = \int^{A \in \dgcatname{A}} \dgmodname{M}(A) \otimes_k \dgmodname{N}(A).
\]
 \end{eg}

Given two dg cateogries $\dgcatname{A}, \dgcatname{B}$, then there is an isomorphism of dg-categories $\dgmod(\dgcatname{A}^{\op} \otimes \dgcatname{B}) \simeq \Fundg(\dgcatname{A}, \dgmod(\dgcatname{B}))$.
So, for an object $E \in \dgmod(\dgcatname{A}^{\op} \otimes \dgcatname{B})$, we have a corresponding dg functor $\Psi_E: \dgcatname{A} \rightarrow \dgmod(\dgcatname{B})$.
Conversely, let $F: \dgcatname{A} \rightarrow \dgcatname{B}$ be a dg functor.
Then, there exists a unique $E \in \dgmod(\dgcatname{A}^{\op} \otimes \dgcatname{B})$ such that $\Psi_E = F$.

Let $F:\dgcatname{A} \rightarrow \dgmod(\dgcatname{B})$ be a dg functor corresponding to $E \in \dgmod(\dgcatname{A}^{\op} \otimes \dgcatname{B})$.
Following \cite[Section 3.1]{cananoco2015internal}, we define the extension and restriction of $F$.
The \emph{extension} of $F$ is the dg functor $F^*: \dgmod(\dgcatname{A}) \rightarrow \dgmod(\dgcatname{B})$ defined by $F^*(-):= - \otimes_\dgcatname{A} E $.
The \emph{restriction} of $F$ is the dg functor $F_*: \dgmod(\dgcatname{A}) \rightarrow \dgmod(\dgcatname{B})$ defined by $F_*(-):= \Hom_\dgcatname{A}(F(-),E)$.
The dg functors $(F^\ast, F_\ast)$ are adjoint to each other: for $\dgmodname{M} \in \dgmod(\dgcatname{A})$ and $\dgmodname{N} \in \dgmod(\dgcatname{B})$, we have the natural isomorphism
\[
\Hom_{\dgmod(\dgcatname{B})}(F^\ast(\dgmodname{M}), \dgmodname{N}) \simeq \Hom_{\dgmod(\dgcatname{A})}(\dgmodname{M}, F_\ast(\dgmodname{N})).
\]
Let $\dgmodname{M} $ 
be a dg $\dgcatname{A}$-$\dgcatname{B}$-bimodule.
Then, $\dgmodname{M}$ is \emph{right quasi-respresentable} or a \emph{quasi-functor} if the dg $\dgcatname{B}$-module $\dgmodname{M}(A, -)$ is quasi-representable for all $A \in \dgcatname{A}$, i.e., $\dgmodname{M}(A, -)$ is isomorphic to a representable dg $\dgcatname{B}$-module in $D(\dgcatname{B})$ for all $A \in \dgcatname{A}$.

\subsection{Noncommutative projective geometry}
\label{subsec:ncproj}
We recall basic notions about noncommutative projective geometry introduced in \cite{artin1994noncommutative}.

Let $A= \bigoplus_{i \in \Zbb} A_i$ be a $\Zbb$-graded $k$-algebra.
We denote by $A^{\op}$ the \emph{opposite algebra} of $A$ and define
the \emph{enveloping algebra} $A^{\en} = A \otimes_k A^{\op}$.
We denote by $\Gr(A)$ the category of graded right $A$-modules with morphisms the $A$-module homomorphisms of degree $0$ and by $\gr(A)$ the full subcategory of finitely generated graded right $A$-modules.
For a graded right $A$-module $M$, we define the Matlis dual $M' \in \Gr(A^{\op})$ by $M'_i:= \Hom_k(M_{-i},k)$.
For an integer $n \in \Zbb$, we define the \emph{truncation} $M_{\geq n} (\text{resp. }M_{\leq n}) \in \Gr(A)$ by $M_{\geq n} = \bigoplus_{i \geq n} M_i$ (resp. $M_{\leq n} = \bigoplus_{i \leq n} M_i$) and the \emph{shift} $M(n) \in \Gr(A)$ by $M(n)_i = M_{i+n}$.
We say that $M$ is \emph{right} (respectively, \emph{left}) \emph{bounded} if there exists $n \in \Zbb$ such that $M_{\geq n} = 0$ (respectively, $M_{\leq n} = 0$).
We say that $M$ is \emph{bounded} if it is both right and left bounded.
For $M,N \in \Gr(A)$, we write the graded vector space 
\[
\Homint_A(M,N):= \bigoplus_{n \in \Zbb} \Hom_{\Gr(A)}(M,N(n)), \quad \Extint_A^i(M,N):= \bigoplus_{n \in \Zbb} \Ext_{\Gr(A)}^i(M,N(n)) \ (i \geq 0).
\]
An $\Nbb$-graded $k$-algebra $A$ is called noetherian if $A$ is right and left noetherian. 

Let $A$ be an $\Nbb$-graded $k$-algebra.
If $A_0=k$, then $A$ is called \emph{connected}. 
If $A_i$ is a finite-dimensional $k$-vector space for all $i \in \Nbb$, then $A$ is called \emph{locally finite}.
In particular, if $A$ is right noetherian and $A_0$ is finite dimensional, then $A$ is locally finite.
We often refer to a connected $\Nbb$-graded $k$-algebra simply as a connected graded $k$-algebra when there is no confusion.
We assume that $A$ is a noetherian $\Nbb$-graded $k$-algebra.
For $M \in \Gr(A)$, an element $m \in M$ is called \emph{torsion} if there exists $n \in \Nbb$ such that $mA_{\geq n} = 0$.
When any element of $M$ is torsion, we say that $M$ is a \emph{torsion module}.
Denote by $\Tor(A)$ (respectively, $\tor(A)$) the full subcategory of $\Gr(A)$ (respectively, $\gr(A)$) consisting of torsion modules.
We write $\QGr (A)$ (respectively, $\qgr (A)$) for the quotient category of $\Gr(A)$ (respectively, $\gr(A)$) by $\Tor(A)$.
We write $\pi_A$ for the projection functor and $\omega_A$ for its right adjoint.
We also denote the composition $\Q_A:=\omega_A\pi_A$.
We denote by $\tau_A$ the functor that takes $M \in \Gr(A)$ to the submodule of $M$ consisting of torsion elements.
Note that we have isomorphisms of graded right $A$-modules (cf. \cite[page 246]{artin1994noncommutative})
\begin{align*}
    \Q_A(M) \simeq \lim_{n \to \infty} \Homint(A_{\geq n}, M), \quad 
    \tau_A(M) \simeq \torfunct{A}(M):= \lim_{n \to \infty} \Homint(A/A_{\geq n}, M).
\end{align*}
We define the Serre twisting sheaf $\qgrobjname{O}_A(n) := \pi_A(A(n))$ for $n \in \Zbb$.

\begin{dfn}[{\cite[Definition 3.2]{artin1994noncommutative}}]
    \label{def:chi-conditions}
Let $A$ be a locally fnite noetherian $\Nbb$-graded $k$-algebra.
We say that $A$ satisfies \emph{the $\chi$-condition} if for any $M \in \gr(A)$, $\Extint^i(A_0,M)$ is bounded for all $i \in \Zbb$. 
We say that $A$ satisfies \emph{$\chi^{\op}$-condition} if $A^{\op}$ satisfies the $\chi$-condition.

\end{dfn}

\begin{dfn}[{\cite[Definition 2.6]{mori2021acategorical}}]
    \label{def:EF condition}
    For a locally finite $\Nbb$-graded algebra $A$, we say that $A$ satisfies the condition (EF) if every dimensional graded right $A$-module $M$ is graded right coherent (i.e., $M$ is finitely generated and every finitely generated graded submodule of $M$ is finitely presented). 
\end{dfn}

\begin{rmk}[{\cite[Page 501]{mori2021acategorical}}]
    \label{rmk:EF condition}
    When $A$ is a connected graded $k$-algebra, the condition (EF) is equivalent to the condition that $A$ is Ext-finite.
    In addition, A right coherent $\Nbb$-graded $k$-algebra $A$ satisfies the condition (EF).
    If a locally finite $\Nbb$-graded $k$-algebra $A$ satisfies the condition (EF), then $\Rtorfunct{A}(-)$ commutes with direct limits (\cite[Lemma 2.17]{mori2021acategorical}).
    By using this property, it is proved that the local duality theorem holds for a locally finite $\Nbb$-graded $k$-algebra $A$ satisfying the condition (EF) (\cite[Theorem 2.18]{mori2021acategorical}).
    In additon, we can show that $\R^i\torfunct{A}(M)=0$ for all $i>0$ and a graded torsion right $A$-module $M$ by also using \cite[Proposition 3.1 (5)]{artin1994noncommutative} (cf. \cite[Lemma 5.10]{mori2021local}, \cite[Lemma 4.4]{vandenbergh1997existence})
\end{rmk}

\begin{dfn}[cf. {\cite[Definitions 6.1 and 6.2]{vandenbergh1997existence}, \cite[Definitions 3.3 and 4.1]{yekutieli1992dualizing}}]
    \label{def:dualizing complex}
    Let $A$ be a noetherian locally finite $\Nbb$-graded $k$-algebra.
    A \emph{dualizing complex} for $A$ is a complex $R \in D(\Gr(A^{\en}))$ such that 
    \begin{enumerate}
        \item $R$ has finite injective dimension over $A$ and $A^{\op}$,
        \item the cohomologies of $R$ are finitely generated as both right and left $A$-modules,
        \item the natural morphisms $A \rightarrow \RHom_{A}(R,R)$ and $A \rightarrow \RHom_{A^{\op}}(R,R)$ are isomorphisms in $D(\Gr(A^{\en}))$.
    \end{enumerate}
    $R$ is called \emph{balanced} if $\Rtorfunct{A}(R) \simeq A'$ and $\Rtorfunct{A^{\op}}(R) \simeq A'$ in $D(\Gr(A^{\en}))$.
\end{dfn}

\begin{thm}[cf. {\cite[Theorem 6.3]{vandenbergh1997existence}}]
    \label{thm:balanced dualizing complex}
    A noetherian locally finite $\Nbb$-graded $k$-algebra $A$ has a balanced dualizing complex if and only if
    \begin{enumerate}
        \item $A$ satisfies the $\chi$-condition and $\chi^{\op}$-condition,
        \item $\torfunct{A}$ and $\torfunct{A^{\op}}$ have finite cohomologial dimension.
    \end{enumerate}
    In this case, the balanced dualizing complex $R_A$ is given by $R_A= \Rtorfunct{A}(A)'$.
\end{thm}

\begin{proof}
    \label{proof:balanced dualizing complex}
    In \cite{vandenbergh1997existence}, the author proved the existence of a balanced dualizing complex for a noetherian connected graded $k$-algebra.
    However, one can check that the proof works for a noetherian locally finite $\Nbb$-graded $k$-algebra as stated in \cite[Proof of Lemma 3.5]{reyes2014skew}.
    In particular, Ext-finiteness of a connected graded $k$-algebra is replaced by the condition (EF) for a locally finite $\Nbb$-graded $k$-algebra.
    See also \cite{wu2001dualizing} or \cite{chan2002pre}.
    To be sure, we provide further details below.
    In \cite[Theorem 4.2 (3)]{yekutieli1997serre}, the authors proved the conditions (1) and (2) from the existence of a balanced dualizing complex for a connected graded $k$-algebra.
    In their proof, the key theorem is \cite[Theorem 4.18]{yekutieli1992dualizing}.
    The proof of \cite[Theorem 4.18]{yekutieli1992dualizing} works for a locally finite $\Nbb$-graded $k$-algebra and we can prove the conditions (1) and (2) from the existence of a balanced dualizing complex for a locally finite $\Nbb$-graded $k$-algebra in the same way.
\end{proof}

For an abelian category $\abcatname{C}$, the \emph{global dimension} $\gldim(\abcatname{C})$ of $\abcatname{C}$ is the maximal integer $n$ such that there exist $X,Y \in \abcatname{C}$ with $\Ext^n(X,Y) \neq 0$.

In the same way as in \cref{thm:balanced dualizing complex}, we can prove \cite[Theorem A.4]{de2004ideal} for a noetherian locally finite $\Nbb$-graded $k$-algebra (see also \cite[Proof of Theorem 4.12]{minamoto2011structure}).

\begin{thm}[cf. {\cite[Theorem A.4]{de2004ideal}}]
    Let $A$ be a noetherian locally finite $\Nbb$-graded $k$-algebra.
    We assume that $A$ has a balanced dualizing complex $R_A$ and the global dimension of $\qgr (A)$ is finite.
    Then, the functor 
    \begin{align*}
        \Scal_A: D^b(\qgr (A)) &\rightarrow D^b(\qgr (A)), \\
        \pi_A(M) &\mapsto \pi_A(M \otimes_A^{\LD} R_A)[-1]
    \end{align*}
    is the Serre functor for $D^b(\qgr (A))$.
\end{thm}

An \emph{algebraic triple} consists of a $k$-linear category $\abcatname{C}$, an object $\Ocal \in \abcatname{C}$ and a $k$-linear auto equivalence $s$ of $\abcatname{C}$.
In this case, we also say that $(\Ocal, s)$ is an \emph{algebraic pair} for $\abcatname{C}$.
For two algebraic triples $(\abcatname{C}, \Ocal, s), (\abcatname{C}', \Ocal', s')$, a morphism of algebraic triples $(F, \theta, \mu):(\abcatname{C}, \Ocal, s) \rightarrow (\abcatname{C}', \Ocal', s')$ consists of a $k$-linear functor $F: \abcatname{C} \rightarrow \abcatname{C}'$, an isomorphism $\theta: F(\Ocal) \rightarrow \Ocal'$ and a natural isomorphism $\mu: F \circ s \rightarrow s' \circ F$.
We say that $(F, \theta, \mu)$ is an \emph{isomorphism} if $F$ is an equivalence of categories.
For an algebraic triple $(\abcatname{C}, \Ocal, s)$, we define a $\Zbb$-graded $k$-algebra $B(\abcatname{C}, \Ocal, s)$ by 
\[
B(\abcatname{C}, \Ocal, s):= \bigoplus_{n \in \Zbb} \Hom_\abcatname{C}(\Ocal, s^n(\Ocal)).
\]

In this paper, the main example of algebraic triples is $(\qgr (A), \Ocal_A, (1)_A)$, where $A$ is a noetherian locally finite $\Nbb$-graded $k$-algebra, $\Ocal_A=\pi_A(A)$ and $(1)_A$ is the degree shift functor.
\begin{dfn}[{\cite[Definition 3.3, 3.4]{mori2021acategorical}}]
\label{def:bimodules}
Let $\abcatname{C}$ be an abelian category.
A \emph{bimodule} $\abobjname{M}$ over $\abcatname{C}$ is an adjoint pair of functors from $\abcatname{C}$ to itself with suggestive notation
\[
\abobjname{M} = (-\otimes \abobjname{M}, \Hom(\abobjname{M}, -)).
\]
A bimodule $\abobjname{M}$ is called \emph{invertible} if $-\otimes \abobjname{M}$ is an equivalence of $\abcatname{C}$.
An invertible bimodule $\abobjname{M}$ is called a \emph{canonical bimodule} for $\abcatname{C}$ if there exists $n \in \Zbb$ such that the induced autoequivalence $-\otimes \abobjname{M}[n]$ of $D^b(\abcatname{C})$ is a Serre functor for $D^b(\abcatname{C})$.
We denote the canonical bimodule by $\canmod_\abcatname{C}$.
When $\abcatname{C} = \qgr (A)$ for a noetherian locally finite $\Nbb$-graded $k$-algebra $A$, we often write the canonical bimodule simply by $\canmod_A$.
\end{dfn}


\begin{rmk}[See also {\cite[Remark 3.5]{mori2021acategorical}}]
\label{rmk:canonical bimodule}
If $\abcatname{C}$ has a canonical bimodule and $- \otimes \canmod_\abcatname{C} [n]$ is the Serre functor for $D^b(\abcatname{C})$, then $\gldim(\abcatname{C}) = n < \infty$.
\end{rmk}

In our study, we need to consider noncommutative projective schemes associated to $\Nbb^2$-graded algebras.
Let $C:= \bigoplus_{i,j \geq 0}C_{i,j} $ be an $\Nbb^2$-graded $k$-algebra.
We introduce notions related to $\Nbb^2$-graded algebras and noncommutative projective schemes associated to $\Nbb^2$-graded algebras following \cite{ballard2021kernels}.
$C$ is called \emph{connected} if $C_{0,0}=k$.
$C$ is called \emph{locally finite} if $C_{i,j}$ is a finite-dimensional $k$-vector space for all $i,j \in \Nbb$.
We denote by $\BiGr(C)$ the category of $\Zbb^2$-graded right $C$-modules.
We often say graded for $\Zbb$-graded and bigraded for $\Zbb^2$-graded.
Let $M \in \BiGr(C)$.
We denote by $C^{\op}$ the opposite algebra of $C$ and define the \emph{enveloping algebra} $C^{\en} = C \otimes_k C^{\op}$.
For $n_1,n_2 \in \Zbb$, we define the truncation $M_{\geq n_1, \geq n_2}$ (respectively, $M_{\leq n_1, \leq n_2}$) by $M_{\geq n_1, \geq n_2} = \bigoplus_{i \geq n_1, j \geq n_2} M_{i,j}$ (respectively, $M_{\leq n_1, \leq n_2} = \bigoplus_{i \leq n_1, j \leq n_2} M_{i,j}$) and the shift $M(n_1, n_2)$ by $M(n_1, n_2)_{i,j} = M_{i+n_1, j+n_2}$.

Our main example of $\Nbb^2$-graded $k$-algebras is the tensor product $A \otimes_k B$ of two locally finite $\Nbb$-graded $k$-algebras $A,B$.
Note that when $A,B$ are noetherian locally finite $\Nbb$-graded $k$-algebras, then in general, $A \otimes_k B$ is not noetherian, but only finitely generated as a $k$-algebra (cf. \cite[Remark 15.1.30]{yekutieli2019derived}).
We have a canonical way to produce graded $A$-modules and graded $B$-modules from a bigraded $A \otimes_k B$-module: For fixed $u,v \in \Zbb$,
\begin{align*}
    (-)_{*, v}:\BiGr(A \otimes_k B) \rightarrow \Gr(A), \quad
    M \mapsto M_{*,v}:= \bigoplus_{i \in \Nbb} M_{i,v}, \\
    (-)_{u, *}:\BiGr(A \otimes_k B) \rightarrow \Gr(B), \quad
    M \mapsto M_{u,*}:= \bigoplus_{j \in \Nbb} M_{u,j}.
\end{align*}
Notice that the forgetful functors $\U_A, \U_B$ also induces the functors $\U_A, \U_B$ from $\BiGr(A \otimes_k B)$ to $\Gr(A), \Gr(B)$, respectively:
\begin{align*}
    \U_A:\BiGr(A \otimes_k B) \rightarrow \Gr(A), \quad
    M \mapsto \U_A(M) := \bigoplus_{u \in \Zbb} M_{u,*}, \\
    \U_B:\BiGr(A \otimes_k B) \rightarrow \Gr(B), \quad
    M \mapsto \U_B(M) := \bigoplus_{v \in \Zbb} M_{*,v}.
\end{align*}
If $A,B$ are finitely generated, then we also have functors:
\begin{align*}
    \Q_A':\BiGr(A \otimes_k B) \rightarrow \BiGr(A \otimes_k B), \quad 
    M \mapsto \Q_A'(M) := \bigoplus_{v \in \Zbb} \Q_A(M_{*,v}), \\
    \Q_B':\BiGr(A \otimes_k B) \rightarrow \BiGr(A \otimes_k B), \quad
    M \mapsto \Q_B'(M) := \bigoplus_{u \in \Zbb} \Q_B(M_{u,*}).
\end{align*}
Note that for a bigraded module $M$, we have the following isomorphisms of graded $A$-modules and graded $B$-modules (\cite[Lemma 3.17]{ballard2021kernels}):
\begin{align*}
    \Q_A'(M) \simeq \Q_A \circ \U_A(M), \quad
    \Q_B'(M) \simeq \Q_B \circ \U_B(M).
\end{align*}
In particular, $\Q_A \circ \U_A(M)$ and $\Q_B \circ \U_B(M)$ have bigraded $A \otimes_k B$-module structures.
A similar argument holds for $\tau_A$ and $\tau_B$.

Let $A, B$ be finitely generated locally finite $\Nbb$-graded $k$-algebras and $M$ be a bigraded $A \otimes_k B$-module.
Then, $m \in M$ is called \emph{torsion} if there exist $n_1, n_2 \in \Zbb$ such that $m (A \otimes_k B)_{\geq n_1, \geq n_2} =0$.
We call $M \in \BiGr(A \otimes_k B)$ a \emph{torsion module} if every element of $M$ is torsion.
We denote by $\Tor(A \otimes_k B)$ the full subcategory of $\BiGr(A \otimes_k B)$ consisting of torsion modules.

\begin{lem}[cf. {\cite[Lemma 3.19]{ballard2021kernels}}]
    \label{lem:tor is Serre}
    $\Tor(A \otimes_k B)$ is a Serre subcategory of $\BiGr(A \otimes_k B)$.
\end{lem}
\begin{proof}
    \label{proof:tor is Serre}
    In \cite[Lemma 3.19]{ballard2021kernels}, the author proved the lemma for connected $\Nbb^2$-graded $k$-algebras.
    We can prove the lemma for locally finite $\Nbb^2$-graded $k$-algebras by the same argument.
\end{proof}

We denote by $\QBiGr(A \otimes_k B)$ the quotient category of $\BiGr(A \otimes_k B)$ by $\Tor(A \otimes_k B)$.
We also denote by $\pi_{A \otimes_k B}$ the projection functor and $\omega_{A \otimes_k B}$ the right adjoint of $\pi_{A \otimes_k B}$.
We also denote the composition $\Q_{A \otimes_k B}:=\omega_{A \otimes_k B} \circ \pi_{A \otimes_k B}$.

\section{Fourier-Mukai functors in noncommutative projective geometry}
\label{sec:FM in ncproj}

Let $A$ be an $\Nbb$-graded $k$-algebra.
We denote by $\Cdg(\Gr(A))$ the dg category of complexes of graded right $A$-modules.
We associate to a $\Zbb$-graded $k$-algebra $A$ the category $\dgcatname{A}$ with objects the integers $\Zbb$ and morphisms given by $\Hom_{\dgcatname{A}}(i,j) = A_{j-i}$.
We naturally regard $\dgcatname{A}$ as a dg category.
Then, we have an equivalence of dg categories 
\[
\Cdg(\Gr(A)) \simeq \dgmod(\dgcatname{A})
\] 
from \cite[Lemma 3.50]{ballard2021kernels}.
A similar argument holds for the $\Zbb^2$-graded case.
Let $M \in \Gr(A)$ and $N \in \Gr(A^{\op})$.
If we think of $M,N$ as right and left dg $\dgcatname{A}$-modules by using the above equivalence, respectively, then the tensor product $M \otimes_{\dgcatname{A}} N$ of dg modules is the same as the $\Zbb$-algebra tensor product (see also \cite[Section 4.1]{mori2021local}).
In this case, we denote by $\Mod(\dgcatname{A})$ the abelian category of right $\dgcatname{A}$-modules.
We also denote by $\Tor(\dgcatname{A})$ the full subcategory of $\Mod(\dgcatname{A})$ consisting of torsion modules.
$\Tor(\dgcatname{A})$ is a Serre subcategory of $\Mod(\dgcatname{A})$ and we denote by $\QMod(\dgcatname{A})$ the quotient category of $\Mod(\dgcatname{A})$ by $\Tor(\dgcatname{A})$.
It is known that $\Gr(A), \Tor(A)$ and $\QGr (A)$ are equivalent to $\Mod(\dgcatname{A}), \Tor(\dgcatname{A})$ and $\QMod(\dgcatname{A})$, respectively (\cite[Page 885,886]{lunts2010uniqueness} or \cite[Lemma 2.17]{mori2023categorical}).
We also note that we have the equivalence
\[
D(\QGr (A)) \simeq D(\dgcatname{A})/D_{\Tor}(\dgcatname{A}),
\]
where $D_{\Tor}(\dgcatname{A}) \subset D(\dgcatname{A})$ is the full subcategory consisting of the objects whose cohomologies are in $\Tor(\dgcatname{A})$ (\cite[Lemma 7.2]{lunts2010uniqueness}).
Let $D_{\dg}(\QGr (A)) := C_{\dg}(\QGr (A)) / \Acydg (\QGr (A))$, where $\Acydg (\QGr (A))$ is the dg subcategory of acyclic complexes in $C_{\dg}(\QGr (A))$. 
Then, we have the quasi-equivalences of dg categories
\[
    D_{\dg}(\QGr (A)) \simeq D_{\dg}(\dgcatname{A})/D_{\Tor, \dg}(\dgcatname{A}),
\]
where $D_{\Tor, \dg}(\dgcatname{A})$ is the dg subcategory of $D_{\dg}(\dgcatname{A})$ consisting of the objects whose cohomologies are in $\Tor(\dgcatname{A})$.
We define $\Perfdg(\QGr (A))$ to be the full dg subcategory of $D_{\dg}(\QGr (A))$ consisting of the compact objects in $H^0(D_{\dg}(\QGr (A)))$.

\begin{dfn}[{\cite[Definition A.1]{Bondal2001reconstruction}, \cite[Definition 9.4]{lunts2010uniqueness}}]
    \label{def:BO-ample}
    Let $\abcatname{C}$ be a $k$-linear exact category.
    Let $\{\abobjname{P}_i\}_{i \in \Zbb}$ be a sequence of objects in $\abcatname{C}$.
    Then, the sequence $\{\abobjname{P}_i\}_{i \in \Zbb}$ is called \emph{BO (Bondal-Orlov)-ample} if there exists an exact embedding $\abcatname{C} \subset \abcatname{C}'$ in an abelian category $\abcatname{C}'$ such that 
    \begin{enumerate}
        \item a morphism in $\abcatname{C}$ is an admissible epimorphism if and only if it is an epimorphism in $\abcatname{C}'$ (we call this condition (EPI)),
        \item for every $\abobjname{M} \in \abcatname{C'}$, there exists $n_0$ such that for all $n < n_0$, the following hold:
        \begin{enumerate}[label=(\alph*)]
        \item there is an epimorphism $\abobjname{P}_n^{\oplus m_n} \rightarrow \abobjname{M}$ for some $m_n \in \Nbb$,
        \item $\Ext^i_\abcatname{C'}(\abobjname{P}_n, \abobjname{M})= 0$ for all $i > 0$,
        \item $\Hom_\abcatname{C'}(\abobjname{M}, \abobjname{P}_n)=0$.
        \end{enumerate}
    \end{enumerate}
    \end{dfn}

    Let $\Ddg(\BiGr(A^{\op} \otimes_k B))$ be a dg enhancement of $D(\QBiGr(A \otimes_k B))$.
    The following theorem is essential for our study.
    \begin{thm}[cf. {\cite[Theorem 4.15]{ballard2021kernels}}]
        \label{thm:dg morita}
        Let $A,B$ be noetherian locally finite $\Nbb$-graded $k$-algebras.
        If $A,B$ have balanced dualizing complexes, then we have an equivalence of dg categories
        \begin{align*}
            \Ddg(\QBiGr(A^{\op} \otimes_k B)) &\overset{\sim}{\longrightarrow} \RHomint_c (\Ddg(\QGr (A)), \Ddg(\QGr (B))), \\
            \qgrobjname{E} &\longmapsto \Phi^{\dg}_{\qgrobjname{E}} (-) := \pi_B(\Romega_A(-)\otimes_{\Acal}^{\LD} \Romega_{A^{\op}\otimes_k B}(\qgrobjname{E})),
        \end{align*}
        where $\RHomint_c$ denotes the dg category formed by the direct sums preserving quasi-functors.
    \end{thm}

    \begin{proof}
        \label{proof:dg morita}
        In \cite{ballard2021kernels}, the author proved the theorem for noetherian  connected graded $k$-algebras.
        However, the same proof works for noetherian locally finite $\Nbb$-graded $k$-algebras.
       This situation is exactly the same as in the proof of \cref{thm:balanced dualizing complex}.
    \end{proof}

    \begin{rmk}
        \label{rmk:dg morita}
        For noetherian locally finite $\Nbb$-graded $k$-algebras $A,B$, the condition that $A,B$ have balanced dualizing complexes is the same as the condition that the pair $(A,B)$ is a delightful couple in the sense of \cite{ballard2021kernels}.
    \end{rmk}

    \begin{nota}
    \label{notation: associated fm functor}
    Let $A,B$ be as in \cref{thm:dg morita}.
    We denote by $\Phi_{\qgrobjname{E}}$ the associated (Fourier-Mukai) functor between $D(\QGr (A))$ and $D(\QGr (B))$.
    \end{nota}

    Let $A$ be a right noetherian locally finite $\Nbb$-graded $k$-algebra.
    Then, we have the following diagram:
    \[
    \begin{tikzcd}
        D^b(\Locfree(\qgr (A))) \arrow[r, "\iota"] & D^b(\qgr (A)) \arrow[d,"\sim" sloped] & & \Perf(\QGr (A)) \arrow[d, "\sim" sloped] \\
        & D^b_{\qgr (A)}(\QGr (A)) \arrow[r, hook, "i_1"] & D(\QGr (A)) & D(\QGr (A))^c \arrow[l, hook', "i_2" '], \\
    \end{tikzcd}
    \] 
    where $\Locfree(\qgr (A)):=\Perf(\QGr (A)) \cap \qgr (A) \subset \qgr (A)$ is an exact category, the vertical arrows are equivalences (\cite[Lemma 4.3.3]{bondal2002generators}, \cref{eg:perfect dg module}), the horizontal arrows $i_1, i_2$ in the bottom row are fully faithful and $\iota:D^b(\Locfree(\qgr (A))) \rightarrow D^b(\qgr (A))$ is a natural functor.
    Note that if $\gldim(\QGr (A)) < \infty$, then we have an equivalence $D^b(\qgr (A)) \simeq \Perf(\QGr (A))$ (\cite[Lemma 4.3.2]{bondal2002generators}).

    \begin{rmk}
        In \cite{bondal2002generators}, the authors proved \cite[Lemma 4.3.2, Lemma 4.3.3]{bondal2002generators} and \cite[Lemma 4.2.2]{bondal2002generators} (we use this lemma below) only for connected graded $k$-algebras.
        However, the same proofs work for locally finite $\Nbb$-graded $k$-algebras because we can use the results in \cref{rmk:EF condition}.
    \end{rmk}
   
The following lemma is basic.

\begin{lem}
\label{lem:BO-ample 1}
Let $A$ be a right noetherian locally finite $\Nbb$-graded $k$-algebra.
Then, we have the following.
\begin{enumerate}
    \item $\iota:D^b(\Locfree(\qgr (A))) \rightarrow D^b(\qgr (A))$ is fully faithful.
    \item $D^b(\Locfree(\qgr (A)))$ is equivalent to $\Perf(\QGr (A))$.  
\end{enumerate}
\end{lem}

\begin{proof}
    \label{proof:lem of BO-ample 1}
(1) \ From \cite[Remark 9.5]{lunts2010uniqueness}, it is enough to show that 
\begin{enumerate}[label=(\alph*)]
    \item (EPI) holds in $\Locfree(\qgr (A)) \subset \qgr (A)$ and 
    \item for any $\qgrobjname{M} \in \qgr (A)$, there exists an epimorphism $\qgrobjname{E} \twoheadrightarrow \qgrobjname{M}$ for some $\qgrobjname{E} \in \Locfree(\qgr (A))$.
\end{enumerate}
Let $\qgrobjname{E}_1, \qgrobjname{E}_2 \in \Locfree(\qgr (A))$.
Let $f: \qgrobjname{E}_1 \rightarrow \qgrobjname{E}_2$ be an epimorphism in $\qgr (A)$.
Then, $\Ker(f)[1] \simeq \Cone(f)$ in $D(\QGr (A))$.
Since $D(\QGr (A))^c$ is a full triangulated subcategory of $D(\QGr (A))$, we have $\Ker(f) \in D(\QGr (A))^c$.
This means that (a) holds.
(b) follows from the fact $\qgrobjname{O}_A(n)$ is a compact object for any $n \in \Zbb$ and \cite[Proposition 4.4 (1)]{artin1994noncommutative}.

(2) \ Let $\qgrobjname{F} := (\cdots \rightarrow 0 \rightarrow \qgrobjname{F}^a \rightarrow \cdots \rightarrow \qgrobjname{F}^b \rightarrow 0 \rightarrow \cdots) \in D^b(\Locfree(\qgr (A)))$, where $a \leq b$ and $\qgrobjname{F}^a, \qgrobjname{F}^b \neq 0$.
Let $\sigma_{>i}\qgrobjname{F}:= (\cdots \rightarrow 0 \rightarrow \qgrobjname{F}^{i+1} \rightarrow \qgrobjname{F}^{i+2} \rightarrow \cdots)$ be the $i$-th stupid truncation of $\qgrobjname{F}$.
Then, we have a distinguished triangle
\[
\sigma_{>a}\qgrobjname{F} \rightarrow \qgrobjname{F} \rightarrow \qgrobjname{F}^a \rightarrow \sigma_{>a}\qgrobjname{F}[1].
\]
So, if $\sigma_{>a}\qgrobjname{F} \in D(\QGr (A))^c$, then $\qgrobjname{F} \in \Perf(\QGr (A))$.
By induction, we have any $\qgrobjname{F} \in D^b(\Locfree(\qgr (A)))$ is compact in $D(\QGr (A))$.
In particular, $i_1 \circ \iota$ factors through $D(\QGr (A))^c$.

From (1), it is enough to show that $i_1 \circ \iota:D^b(\Locfree(\qgr (A))) \rightarrow D(\QGr (A))^c$ is essentially surjective.
To prove this, we use the idea of the proof of \cite[Lemma 36.37.1]{stacks-project}.
Firstly, by \cite[Lemma 4.2.2]{bondal2002generators}, we have $D(\QGr (A))^c = \langle \qgrobjname{O}_A(n) \rangle_{n \in \Zbb}$, i.e., $D(\QGr (A))^c$ is classically generated by $\{\qgrobjname{O}_A(n)\}_{n \in \Zbb}$.
So, $D(\QGr (A))^c$ can be thought as a full triangulated subcategory of $D^b(\qgr (A))$.
Let $\qgrobjname{F} = (\cdots \rightarrow 0 \rightarrow \qgrobjname{F}^a \rightarrow \cdots \rightarrow \qgrobjname{F}^b \rightarrow 0 \rightarrow \cdots) \in D(\QGr (A))^c$, where $a \leq b$ and $\qgrobjname{F}^a, \qgrobjname{F}^b \neq 0$.
We prove that there exists $\qgrobjname{E} \in D^b(\Locfree(\qgr (A)))$ and a quasi-isomorphism $\qgrobjname{E} \rightarrow \qgrobjname{F}$ by induction on $b-a$.
When $b-a=0$, the claim is clear.
Assume that the claim holds for $b-a < n$.
There exists an epimorphism $\qgrobjname{E}^b \twoheadrightarrow \qgrobjname{F}^b$ for some $\tilde{\qgrobjname{E}}^b \in \Locfree(\qgr (A))$ by the fact $\qgrobjname{O}_A(n) \in \Locfree(\qgr (A))$ and \cite[Proposition 4.4 (1)]{artin1994noncommutative}.
Then, we have a morphism $\alpha:\tilde{\qgrobjname{E}}^b[-b] \rightarrow \qgrobjname{F}$ of chain complexes of objects in $\qgr (A)$ and a distinguished triangle 
\[ 
\tilde{\qgrobjname{E}}^b[-b] \rightarrow \qgrobjname{F} \rightarrow \Cone(\alpha) 
\]
in $D(\QGr (A))$.
By taking the long exact sequence, it is shown that $H^i(\Cone(\alpha)) \neq 0$ implies $i \in [a, b-1]$.
Considering appropriate standard truncations, we have $\Cone(\alpha) \simeq \qgrobjname{E}'$ for some $\qgrobjname{E}' = (\cdots \rightarrow 0 \rightarrow \qgrobjname{E}'^a \rightarrow \cdots \rightarrow \qgrobjname{E}'^{b-1} \rightarrow 0 \rightarrow \cdots) \in D^b(\qgr (A))$.
From the induction hypothesis, we also have an quasi-isomorphism $\beta:\qgrobjname{E}'' \rightarrow \Cone(\alpha)$ for some $\qgrobjname{E}'' \in D^b(\Locfree(\qgr (A)))$.
Finally, we put $\qgrobjname{E}:= \qgrobjname{E}'' \oplus \tilde{\qgrobjname{E}}^b[-b]$ and the natural morphism $\qgrobjname{E} \rightarrow \qgrobjname{F}$ is a quasi-isomorphism. 
This completes the proof.
\end{proof}

We recall the notion of quasi-Veronese algebras to construct a sequence of objects in $\Locfree(\qgr (A))$ which is BO-ample.

\begin{dfn}[{\cite[Definition 3.7]{mori2013bconstruction}}]
    \label{def:quasi-veronese}
    Let $A$ be an $\Nbb$-graded $k$-algebra and $r \in \Nbb$.
    We define the \emph{r-th quasi-Veronese algebra} $A^{[r]}$ of $A$ by 
    \begin{align*}
        A^{[r]}:= \bigoplus_{i \in \Nbb} 
        \begin{pmatrix}
            A_{ri} & A_{ri+1} & \cdots & A_{ri+l-1} \\
           A_{ri-1} & A_{ri} & \cdots & A_{ri+l-2} \\
           \vdots & \vdots & \ddots & \vdots \\
           A_{ri-l+1} & A_{ri-l+2} & \cdots & A_{ri}.
           \end{pmatrix}.
    \end{align*}
\end{dfn}

We recall basic properties of quasi-Veronese algebras.

\begin{lem}[{\cite[Lemma 3.9]{mori2013bconstruction}, \cite[Remark 3.19, Lemma 3.20]{mizuno2024some}}]
    \label{lem:quasi-veronese}
    Let $A$ be an $\Nbb$-graded $k$-algebra.
    Then, we have the following.
    \begin{enumerate}
        \item the functor 
        \begin{align*}
           \qV: \Gr(A) \rightarrow \Gr\left(A^{[r]}\right), \quad M \mapsto \bigoplus_{i \in \mathbb{Z}} \left(\bigoplus_{j=0}^{r-1} M_{ri-j} \right)
        \end{align*}
        is an equivalence.
        In particular, $\qV(\bigoplus_{i =0}^{r-1} A(i)) = A^{[r]}$.
        \item if $A$ is right noetherian, then $A^{[r]}$ is also right noetherian.
        Moreover, the equivalence in (1) gives an equivalence of categories between $\qgr (A)$ and $\qgr\left(A^{[r]}\right)$.
        \item If $A$ is finitely generated as an $A_0$-algebra, then for large enough $r \in \Nbb$, $A^{[r]}$ is generated by $A^{[r]}_1$ as an $A^{[r]}_0$-algebra.
    \end{enumerate}
\end{lem}

We construct a sequence of objects in $\Locfree(\qgr (A))$ which is BO-ample in the following lemma.

\begin{lem}
    \label{lem:BO-ample 2}
    Let $A$ be a right noetherian locally finite $\Nbb$-graded $k$-algebra.
    If $A$ satisfies the $\chi$-condition and $\R^1\torfunct{A}(A)$ is a finite $A$-module, then for large enough $r \in \Nbb$,
    the sequence $\{\bigoplus_{i=0}^{r-1}\qgrobjname{O}_A(i+rn)\}_{n \in \Zbb}$ is BO-ample in $\Locfree(\qgr (A))$.
\end{lem}

\begin{proof}
    \label{proof:lem of BO-ample 2}
We verify the conditions (1), (2) in \cref{def:BO-ample}.
We have already shown that (EPI) holds in $\Locfree(\qgr (A))$ in the proof of (1) of \cref{lem:BO-ample 1}.

Regarding (a) of (2),  for any $\qgrobjname{M} \in \qgr (A)$, there exists an epimorphism 
\begin{equation*}
    \label{eq:epimorphism 1}
   \psi:\bigoplus_{i=1}^p \qgrobjname{O}_A(-l_i)^{\oplus n_i} \twoheadrightarrow \qgrobjname{M}
\end{equation*}
for some $l_i, n_i \in \Nbb$.
On the other hand, if $r$ is large enough, then $A^{[r]}$ is generated by $A^{[r]}_1$ as an $A^{[r]}_0$-algebra by \cref{lem:quasi-veronese} (3).
So, we have an surjective morphism $A^{[r]}(-1)^{\oplus n} \twoheadrightarrow \left(A^{[r]}\right)_{>0}$ for some $n \in \Nbb$ in $\Gr \left(A^{[r]}\right)$. 
If we put $\widetilde{\qgrobjname{O}}_A(i) = \pi_{A^{[r]}}\left(A^{[r]}(i)\right)$, then we have an epimorphism
\[
    \widetilde{\qgrobjname{O}}_A(-1)^{\oplus n} \twoheadrightarrow \widetilde{\qgrobjname{O}}_A 
\]
in $\qgr \left(A^{[r]} \right)$.
We also put $L_i = \bigoplus_{j=0}^{r-1} \qgrobjname{O}_A(j+ir)$.
Then, $\qV(L_i) = \widetilde{\qgrobjname{O}}_A(i)$ from \cref{lem:quasi-veronese} (1).
So, we have an epimorphism $\varphi: L_{-1}^{\oplus n} \twoheadrightarrow L_0$ from \cref{lem:quasi-veronese} (2).
Here, $\psi$ induces an epimorpshim
\begin{equation*}
    \label{eq:epimorphism 2}
   \psi':\bigoplus_{i=1}^p L_{-k_i}^{\oplus n_i} \twoheadrightarrow \qgrobjname{M}
\end{equation*}
for some $k_i \in \Nbb$.
Thus, by using $\varphi$ and the degree shift functor, we also have an epimorphism $L_{-k}^{\oplus n'} \twoheadrightarrow \qgrobjname{M}$ for some $k,n' \in \Nbb$.
Moreover, for any $n'' < n'$, we have an epimorphism $L_{-k'}^{\oplus n''} \twoheadrightarrow \qgrobjname{M}$ for some $k' \in \Nbb$ by using $\varphi$ and the degree shift functor again.
Hence, (a) holds.

(b) follows from the assumption $A$ satisfies the $\chi$-condition and the noncommutative Serre vanishing theorem \cite[Theorem 7.4]{artin1994noncommutative}.

Regarding (c), for any $\qgrobjname{M} \in \qgr (A)$, we have an epimorphism $f: L_{-k}^{\oplus n} \twoheadrightarrow \qgrobjname{M}$ for some $k,n \in \Nbb$ and 
the long exact sequence
\begin{align*}
    0 \rightarrow \Hom_{\qgr (A)}(\qgrobjname{M}, L_{l}) \rightarrow \Hom_{\qgr (A)}(L_{-k}, L_{l})^{\oplus n} \rightarrow \Hom_{\qgr (A)}(\Ker(f), L_{l}) \rightarrow \cdots
\end{align*}
for any $l \in \Zbb$.
So, it is enough to show that $H^0(\qgr (A),L_l) =\bigoplus_{i=0}^{r-1} H^0(\qgr (A), \qgrobjname{O}_A(i+rl)) =0$ for $l \ll 0$.
Here, we have the canonical exact sequence
\begin{equation}
    \label{eq:canonical exact sequence}
    0 \rightarrow \torfunct{A}(A) \rightarrow A \rightarrow \Q(A) \rightarrow \R^1\torfunct{A}(A) \rightarrow 0 \tag{$\star$}.
\end{equation}
Note that $\Q(A) \simeq \bigoplus_{i \in \Zbb} \Hom(\qgrobjname{O}_A, \qgrobjname{O}_A(i))$.
Since $A$ satisfies the $\chi$-condition, $\R^1\torfunct{A}(A)$ is right bounded (\cite[Proposition 3.14]{artin1994noncommutative}).
Moreover, from the assumption, $\R^1\torfunct{A}(A)$ is a finite $A$-module, so $\R^1\torfunct{A}(A)$ is also left bounded.
This implies that $\Q(A)$ is left bounded and $H^0(\qgr (A), \qgrobjname{O}_A(i)) =0$ for $i \ll 0$.
Hence, $H^0(\qgr (A), L_l) =0$ for $l \ll 0$.
This completes the proof.

\end{proof}

We are now ready to prove the main theorem about Fourier-Mukai functors in noncommutative projective geometry.

\begin{thm}
    \label{thm:re-main2}
    Let $A,B$ be noetherian locally finite $\Nbb$-graded $k$-algebras.
    We assume that $A,B$ have balanced dualizing complexes.
    Let $F: \Perf(\QGr (A)) \rightarrow D(\QGr (B))$ be an exact fully faithful functor.
    Then, there exists an object $\qgrobjname{E} \in D(\QBiGr(A^{\mathrm{op}} \otimes_k B))$ such that 
    \begin{enumerate}
        \item $\Phi_{\qgrobjname{E}}$ is exact fully faithful and $\Phi_\qgrobjname{E}(\qgrobjname{P}) \simeq F(\qgrobjname{P})$ for any $\qgrobjname{P} \in \Perf(\QGr (A))$. \label{thm:re-main2 1}
        \item If $F$ sends a perfect complex to a perfect complex, then the  induced functor $\Phi_\qgrobjname{E}:D(\QGr (A)) \rightarrow D(\QGr (B))$ is fully faithful and $\Phi_\qgrobjname{E}$ sends a perfect complex to a perfect complex. \label{thm:re-main2 2}
        \item If $\R^1\torfunct{A}(A)$ is a finite $A$-module, then $F \simeq \Phi_\qgrobjname{E}$.
        Moreover, if $\Phi_{\qgrobjname{E}} \simeq \Phi_{\qgrobjname{E}'}$ for some $\qgrobjname{E}' \in D(\QBiGr(A^{\mathrm{op}} \otimes_k B))$, then $\qgrobjname{E} \simeq \qgrobjname{E}'$. \label{thm:re-main2 3}
    \end{enumerate}
    \end{thm}
    
    \begin{proof}
        \label{proof:main2}
       Thanks to proving \cref{lem:BO-ample 1} and \cref{lem:BO-ample 2}, the proof is similar to the proof of \cite[Corollary 9.13]{lunts2010uniqueness}.
        We give a sketch of the proof for the reader's convenience.

        \noindent
        \emph{Step 1: Construction of $\qgrobjname{E}$.}

        Here, we describe how to construct $\qgrobjname{E}$.
        Firstly, note that we have $D(\QGr (A)) \simeq D(\dgcatname{A})/D_{\Tor}(\dgcatname{A})$ and $\Perf(\QGr (A)) \simeq (D(\dgmodname{A})/D_{\Tor}(\dgmodname{A}))^c$.
        We also have quasi-equivalences (\cite[Proposition 1.17]{lunts2010uniqueness})
        \begin{align}
        \label{eq:semi-free quasi-equivalence}
           \phi_A:D_{\dg}(\QGr (A)) \xrightarrow{\sim} \sfmod (\Perf_{\dg}(\QGr (A))), \
           \phi_B:D_{\dg}(\QGr (B)) \xrightarrow{\sim} \sfmod (\Perf_{\dg}(\QGr (B))). \tag{$\dagger$}
        \end{align}
        We denote $\dgcatname{C}$ by the full dg subcategory of $\Perf_{\dg}(\QGr (B))$ which consists of the objects in the essential image of $H^0(\phi_B) \circ F$.
        Then, the functor $H^0(\phi_B) \circ F$  induces an equivalence
        \[
        G:\Perf(\QGr (A)) \rightarrow H^0(\dgcatname{C}).
        \]
        By \cite[Theorem 6.4]{lunts2010uniqueness}, we have a quasi-equivalnce 
        \[
        \widetilde{G}:\Perf_{\dg}(\QGr (A)) \rightarrow \dgcatname{C}.
        \]
        This functor induces a quasi-equivalence
        \[
        \widetilde{G}^{*}:\sfmod(\Perfdg(\QGr (A))) \rightarrow \sfmod(\dgcatname{C}).
        \]
        Let $\dgcatname{D}$ be a dg subcategory of $\sfmod(\dgcatname{C})$ that contains $\dgcatname{C}$ and $\Perfdg(\QGr (B))$.  
        We denote by 
        \begin{align*}
            I_1:\dgcatname{C} \hookrightarrow \dgcatname{D}, \quad 
            I_2:\Perfdg(\QGr (B)) \hookrightarrow \dgcatname{D}
        \end{align*}
        the embedding functors.
        Then, we have the extension and the restriction functors
        \begin{align*}
        I_1^*:\sfmod(\dgcatname{C}) \rightarrow \sfmod(\dgcatname{D}), \quad
        {I_2}_*:\sfmod(\dgcatname{D}) \rightarrow \sfmod(\Perfdg(\QGr (B))),
        \end{align*}
        respectively.
        The functors $H^0(I_1^*)$, $H^0({I_2}_*)$ and $H^0(\widetilde{G}^*)$ commute with direct sums.
        Here, we define a quasi-functor 
        \begin{align*}
            \widetilde{F} := \phi_B^{-1} \circ {I_2}_* \circ I_1^* \circ \widetilde{G}^* \circ \phi_A:D_{\dg}(\QGr (A)) \rightarrow D_{\dg}(\QGr (B)). 
        \end{align*}
        This functor commutes with direct sums.
        Thus, there exists an object $\qgrobjname{E} \in D(\QGr(A^{\op} \otimes_k B))$ such that $\Phi_{\qgrobjname{E}}^{\dg} \simeq \widetilde{F}$ by \cref{thm:dg morita}. \par

        \noindent\emph{Step 2: Proof of (1), (2) and (3).}

        (1) \  Because $({I_2}_* \circ I_1^*)|_{\dgcatname{C}}$ is isomorphic to the inclusion functor $\dgcatname{C} \hookrightarrow \sfmod(\Perfdg(\QGr (B)))$, $\Phi_{\qgrobjname{E}}|_{\Perf(\QGr (A))}$ is fully faithful.
        By \cite[Theorem 6.4 (3)]{lunts2010uniqueness}, we have an isomorphism $H^0(\widetilde{G})(\qgrobjname{M}) \simeq G(\qgrobjname{M})$ for any $\qgrobjname{M} \in \Perf(\QGr (A))$.
        Thus, $\Phi_{\qgrobjname{E}}(\qgrobjname{M}) \simeq F(\qgrobjname{M})$ for any $\qgrobjname{M} \in \Perf(\QGr (A))$.
        
        (2) \ If $F$ sends a perfect complex to a perfect complex, then we can take $\dgcatname{D}$ as $\sfmod(\Perfdg(\QGr (B)))$.
        In this case, ${I_2}_*$ is the identity and $I_1^*$ is fully faithful (\cite[Proposition 1.15]{lunts2010uniqueness}).
        This implies the assertion.

        (3) \ From \cite[(2) of Theorem 6.4]{lunts2010uniqueness}, we have an isomorphism of functors 
        \[
            \theta:H^0(\widetilde{G}) \circ \pi \circ \yoneda_{\dgcatname{A}} \rightarrow G \circ \pi \circ \yoneda_{\dgcatname{A}},
        \]
        where the compositions of the above functors given by the following diagram:
        \[
    \begin{tikzcd}
        \dgcatname{A} \arrow[r, "\yoneda_{\dgcatname{A}}"] & D(\dgcatname{A})^c \arrow[r,"\pi"] & D(\dgcatname{A})^c/D_{\Tor}(\dgcatname{A})^c \arrow[hookrightarrow]{r} & (D(\dgmodname{A})/D_{\Tor}(\dgmodname{A}))^c \arrow[r, "G"', "H^0(\widetilde{G})"] & H^0(\dgcatname{C}).
    \end{tikzcd}
    \]
    Let $j: \{L_i\}_{i \in \Zbb} \hookrightarrow \Perf(\QGr (A))$ be the inclusion functor, where $L_i$ is as in the proof of \cref{lem:BO-ample 2}.
    Note that we think of $\{L_i\}_{i \in \Zbb}$ as a full subcategory of $\Locfree (\qgr (A))$ and use \cref{lem:BO-ample 1}.
    Then, we have an isomorphism of functors $j \simeq  G \circ H^0(\widetilde{G}) \circ j$.
    Since $\{L_i\}_{i \in \Zbb}$ is BO-ample in $\Locfree(\qgr (A))$ from \cref{lem:BO-ample 2}, we have $G \circ H^0(\widetilde{G}) \simeq \idfun_{\Perf(\QGr (A))} $ from \cite[Proposition 9.6]{lunts2010uniqueness}.
    This implies that $\Phi_{\qgrobjname{E}} \simeq F$.
    The uniqueness of $\qgrobjname{E}$ follows from \cite[Theorem 5.5]{genenovese2016theuniqueness}.
    \end{proof}

    \begin{rmk}
        \label{rmk:re-main2}
        In the proof of (3) of \cref{thm:re-main2}, to apply \cite[Theorem 5.5]{genenovese2016theuniqueness}, we only use the assumption that $F$ is fully faithful.
    \end{rmk}

    In the proof of \cref{lem:BO-ample 2}, when we check the condition (c), if we only have $H^0(\qgrobjname{O}_A(i)) =0$ for $i \ll 0$, then we can obtain the result of the lemma.
    In particular, we have the following simple version of \cref{thm:re-main2} in the same way as the proof of \cref{thm:re-main2}.
    \begin{prop}
        \label{prop:re-main2}
        Let $A,B$ be noetherian locally finite $\Nbb$-graded $k$-algebras.
        We assume that $A,B$ have balanced dualizing complexes.
        Let $\Perf(\QGr (A)) \rightarrow D(\QGr (B))$ be an exact fully faithful functor.
        We assume that $H^0(\QGr (A),\Ocal_A(m))=0,\ m \ll 0$.
        Then, $F$ is of Fourier-Mukai type and the kernels are unique up to quasi-isomorphism.
    \end{prop}


\section{Bondal-Orlov's reconstruction theorem for noncommutative projective schemes}
\label{sec:BO for ncproj}

\begin{dfn}[{\cite[page 1151, 1155]{ballard2021kernels}}]
    \label{def:R'Q functor}
    We define a functor $\RprimeQ_A$
\begin{align*}
    \RprimeQ_A: D(\BiGr(A \otimes_k B)) &\rightarrow D(\BiGr(A \otimes_k B)),\\
    M &\mapsto \RprimeQ_A(M):= \bigoplus_{v \in \Zbb} \RQ_A(M_{*,v}).
\end{align*}
Similarly, we define a functor $\Rprimetau_A$ by $\Rprimetau_A(M):= \bigoplus_{v \in \Zbb} \Rtau_A(M_{*,v})$.
\end{dfn}

We recall basic properties of $\RQ$, $\Rtau$, $\RprimeQ$ and $\Rprimetau$ in \cite{ballard2021kernels}. 
In the following, we state the corresponding properties in the case of locally finite $\Nbb$-graded algebras.
The proofs are similar to the case of connected graded algebras.

\begin{lem}[cf. {\cite[Lemma 3.35, Proposition 3.36, Proposition 3.41]{ballard2021kernels}}]
\label{lem:basic properties}
Let $A,B$ be finitely generated locally finite $\Nbb$-graded $k$-algebras.
Let $M \in D(\Gr(A))$ and $P \in D(\BiGr(A \otimes_k B))$.
Then, we have the following.
\begin{enumerate}
    \item  We have distinguished triangles
    \begin{align*}
        \Rtau_A(M) \rightarrow &M \rightarrow \RQ_A(M) \quad \text{in} \ D(\Gr(A)), \\
        \Rprimetau_A(P) \rightarrow &P \rightarrow \RprimeQ_A(P) \quad \text{in} \ D(\BiGr(A \otimes_k B)), \\
        \Rtau_{A \otimes_k B}(P) \rightarrow &P \rightarrow \RQ_{A \otimes_k B}(P) \quad \text{in} \ D(\BiGr(A \otimes_k B)). 
    \end{align*} \label{lem:basic properties 2}
    \item  If $A$ satisfies (EF), then $\RQ_A$ and $\Rtau_A$ commute with coproducts. \label{lem:basic properties 3}
    \item If $A,B$ satisfy (EF) and $\tau_A, \tau_B$ have finite cohomological dimension, then we have 
    \begin{equation*}
        \RQ_{A \otimes_k B}(P) \simeq (\RprimeQ_A \circ \RprimeQ_B)(P) \simeq (\RprimeQ_B \circ \RprimeQ_A)(P) 
    \end{equation*}
    in $D(\BiGr(A \otimes_k B))$. \label{lem:basic properties 4}
\end{enumerate}
\end{lem}

\begin{rmk}
    \label{rmk:basic properties}
    In \cite[Proposition 3.41]{ballard2021kernels}, $A,B$ are assumed to be right noetherian.
    However, only Ext-finiteness of $A,B$ is used in the proof of \cite[Proposition 3.41]{ballard2021kernels}.
    So, we replace the assumption by (EF) in \cref{lem:basic properties}.

\end{rmk}

The following lemma is a generalization of \cite[Proposition 3.32]{ballard2021kernels} in the bimodule setting.

\begin{lem}
    \label{lem:bimodule idempotent formula}
    Let $A$ be a finitely generated locally finite $\Nbb$-graded $k$-algebras.
$P \in D(\BiGr(A \otimes_k B))$.
    If $A$ satisfies (EF) and $\tau_A$ has finite cohomological dimension, then we have
    \begin{align*}
        (\Rprimetau_A \circ \Rprimetau_A)(P) &\simeq \Rprimetau_A(P), \\
        (\RprimeQ_A \circ \RprimeQ_A)(P) &\simeq \RprimeQ_A(P) 
    \end{align*}
    in $D(\BiGr(A \otimes_k B))$.
\end{lem}

\begin{proof}
    \label{proof:bimodule idempotent formula}
    From the definition of $\Rprimetau_A$,
    \[
    \Rprimetau_A(P) = \bigoplus_{v \in \Zbb} \Rtau_A(P_{*,v}).
    \]
    For a fibrant object $P \in \BiGr(A \otimes_k B)$, it is shown that $\tau_A(P_{*,v})$ is $\tau_A$-acyclic in \cite[Proposition 3.32]{ballard2021kernels}.
    This gives the first formula of the lemma.
    For the second formula, we have the following triangles in $D(\BiGr(A \otimes_k B))$ from (1) of \cref{lem:basic properties}:
\begin{align*}
    (\Rprimetau_A \circ \Rprimetau_A(P)) \rightarrow &\Rprimetau_A(P) \rightarrow (\RprimeQ_A \circ \Rprimetau_A)(P), \\
    (\RprimeQ_A \circ \Rprimetau_A)(P) \rightarrow &\RprimeQ_A(P) \rightarrow (\RprimeQ_A \circ \RprimeQ_A)(P).
\end{align*}
Now, from the first formula of the lemma, $(\RprimeQ_A \circ \Rprimetau_A)(P)$ is acyclic.
Thus, we get the second formula of the lemma.
\end{proof}

The following lemma is a generalization of \cite[Proposition 4.5]{ballard2021kernels} in the bimodule setting.
\begin{lem}
    \label{lem:bimodule proj formula}
Let $A,B,C$ be finitely generated locally finite $\Nbb$-graded $k$-algebras.
Let $M \in D(\BiGr(B^{\op} \otimes_k A))$ and $N \in D(\BiGr(C^{\op} \otimes_k B))$.
Assume that $\RQ_A$ and $\Rtau_{A}$ commute with coproducts.
Then, we have natural isomorphisms in $D(\BiGr(C^{\op} \otimes_k A))$
\begin{align*}
    \Rprimetau_{A}(N \otimes_{\Bcal}^{\LD} M) \simeq N \otimes_{\Bcal}^{\LD} \Rprimetau_A (M), \\
    \RprimeQ_{A}(N \otimes_{\Bcal}^{\LD} M) \simeq N \otimes_{\Bcal}^{\LD} \RprimeQ_A(M),
\end{align*}
where $\Bcal$ is the dg category associated to $B$.
\end{lem}

\begin{proof}
    \label{proof:bimodule proj formula}
    The proof is similar to the proof of \cite[Proposition 4.5]{ballard2021kernels}.
    However, for the sake of completeness, we give a proof here.
    We only prove the second formula.
    Firstly, note that $N \otimes_{\Bcal}^{\LD} M$ is calculated by using the bar complex 
    \begin{align*}
    \cdots \rightarrow N \otimes_k \Bcal \otimes_k \Bcal \otimes_k M \rightarrow N \otimes_k \Bcal \otimes_k M \rightarrow N \otimes_{k} M \rightarrow 0.
    \end{align*}
    Here, $N \otimes_k \Bcal^{\otimes n} \otimes_k M$ is a dg $\Ccal$-$\Acal$-bimodule defined by
    \begin{align*}
        (N \otimes_k \Bcal^{\otimes n} \otimes_k M)_{i,j} = \bigoplus_{i_0, \cdots, i_n \in \Zbb} N_{i,i_0} \otimes_k \Bcal(i_0,i_1) \otimes_k \cdots \otimes_k \Bcal(i_{n-1},i_n) \otimes_k M_{i_n,j},
    \end{align*}
    where $\Bcal(i,j) := \Hom_{\Bcal}(i,j) \ (i,j \in \Zbb)$ (cf. \cite[Section 6.6]{keller1994deriving}).
    We show that we have a natural morphism 
    \begin{align*}
      \varphi_n:N \otimes_k \Bcal^{\otimes n} \otimes_k \left(\bigoplus_{v \in \Zbb} Q'_A(R_A(M_{*,v}))\right) \rightarrow \bigoplus_{v \in \Zbb} Q'_A(N \otimes_k \Bcal^{\otimes n} \otimes_k R_A(M_{*,v}))
    \end{align*}
    and this map is an isomorphism for any $n \in \Nbb$.
    For simplicity, we put 
    \begin{align*}
        &P_j = R_A(M_{*,j}), \\
        &V_{i_0, \cdots, i_n}^i = N_{i,i_0} \otimes_k \Bcal(i_0,i_1) \otimes_k \cdots \otimes_k \Bcal(i_{n-1},i_n).
    \end{align*}
    Then, we have the following:
    \begin{align*}
        \left(N \otimes_k \Bcal^{\otimes n} \otimes_k \left(\bigoplus_{v \in \Zbb} Q'_A(R_A(M_{*,v}))\right)\right)_{i,j} = \varinjlim_n \bigoplus_{i_0, \cdots, i_n \in \Zbb} V_{i_0, \cdots, i_n}^i \otimes_k \Hom_{\Gr(A)} (A_{\geq n}, P_j(i_n)), \\
        \left(\bigoplus_{v \in \Zbb} Q'_A(N \otimes_k \Bcal^{\otimes n} \otimes_k R_A(M_{*,v}))\right)_{i,j} = \varinjlim_n \Hom_{\Gr(A)}\left(A_{\geq n}, \bigoplus_{i_0, \cdots, i_n \in \Zbb} V_{i_0, \cdots, i_n}^i \otimes_k P_j(i_n)\right).
    \end{align*}
    So, the map 
    \begin{align*}
        \bigoplus_{i_0, \cdots, i_n \in \Zbb} V_{i_0, \cdots, i_n}^i \otimes_k \Hom_{\Gr(A)} (A_{\geq n}, P_j(i_n)) &\rightarrow \Hom_{\Gr(A)}\left(A_{\geq n}, \bigoplus_{i_0, \cdots, i_n \in \Zbb} V_{i_0, \cdots, i_n}^i \otimes_k P_j(i_n)\right), \\
        n \otimes b_{i_0,i_1} \otimes \cdots \otimes b_{i_{n-1},i_n} \otimes f &\mapsto \left(a \mapsto n \otimes b_{i_0,i_1} \otimes \cdots \otimes b_{i_{n-1},i_n} \otimes f(a)\right)
    \end{align*}
    defines the desired levelwise morphisms $\varphi_n$.
    The assumption that $\RQ_A$ commutes with coproducts implies each $\varphi_n$ is an isomorphism because $V_{i_0, \cdots, i_n}^i \otimes_k P_j(i_n)$ is $Q_A$-acyclic (cf. \cite[Lemma 3.27]{ballard2021kernels}).
\end{proof}

The following lemma is a generalization of \cite[Proposition 3.42]{ballard2021kernels} in arbitrary graded finite $A$ and $A^{\op}$-modules.

\begin{lem}
    \label{lem:a good formula fo fin mods}
Let $A$ be a noetherian locally finite $\Nbb$-graded $k$-algebra.
Let $M \in \Gr(A^{\en}) $ be a finitely generated graded right and left $A$-module.
We assume that $A$ has a balanced  dualizing complex.
Then, we have quasi-isomorphisms in $D(\Gr(A^{\en}))$
\begin{align*}
    \RprimeQ_A(M^{\bgm}) \simeq \RQ_{A \otimes_k A^{\op}}(M^{\bgm}) \simeq \RprimeQ_{A^{\op}}(M^{\bgm}),
\end{align*}
where $M^{\bgm}$ is the bimodule defined by $M^{\bgm} := \bigoplus_{i,j \in \Zbb} M^{\bgm}_{i,j} = \bigoplus_{i,j \in \Zbb} M_{i+j}$.
\end{lem}

\begin{proof}
    \label{proof:a good formula for fin mods}
    The proof is similar to the proof of \cite[Proposition 3.42]{ballard2021kernels}.
    However, for the sake of completeness, we give the proof here.


   For symmetry, we only show the first quasi-isomorphism.
    For any $N \in \BiGr(A^{\en})$, the natural morphism
    \begin{align*}
    \eta:\Rprimetau_{A^{\op}}(N) \rightarrow N
    \end{align*}
    is a quasi-isomorphism if and only if the natural morphism
    \begin{align*}
       \eta_{u,*}:\Rtau_{A^{\op}}(N_{u,*}) \rightarrow N_{u,*}
    \end{align*}
    is a quasi-isomorphism for any $u \in \Zbb$.
    In addition, if the $j$-th cohomology of $N_{u,*}$ is right bounded for each $j$, then $\eta_{u,*}$ is a quasi-isomorphism.
    We condier the case $N=\Rprimetau_A(M^{\bgm})$.
    In this case, 
    \[
        N_{u,*} = \bigoplus_{v' \in \Zbb} \Rtau_A \left(M^{\bgm}_{*,v'} \right)_u.
    \]
    Then, we have the following isomorphism in $D(\BiGr(A^{\en}))$
    \begin{align*}
        H^j(N_{u,*}) = \bigoplus_{v' \in \Zbb} \R^j\tau_A \left(M^{\bgm}_{*,v'} \right)_u \simeq \bigoplus_{v' \in \Zbb} \R^j\tau_A (M(v'))_u.
    \end{align*}
    From \cite[Corollary 3.6 (3)]{artin1994noncommutative} and the fact that $\tau_A$ commutes with the degree shift, we have  
    \[
        \R^j\tau_A (M(v'))_u = \R^j\tau_A(M)_{u+v'} = 0
    \]
     for $v' \gg 0$ (we use the assumption that $M$ is finitely generated here).
    So, 
    \[
         \Rtau_{A^{\op}} \left(\bigoplus_{v' \in \Zbb} \Rtau_A \left(M^{\bgm}_{*,v'} \right)_u \right) \rightarrow \bigoplus_{v' \in \Zbb} \Rtau_A \left(M^{\bgm}_{*,v'} \right)_u
    \]
    is a quasi-isomorphism for any $u$.
    We have the distinguished triangle
    \[
    \Rprimetau_{A^{\op}}\Rprimetau_A (M^{\bgm}) \rightarrow \Rprimetau_A (M^{\bgm}) \rightarrow \RprimeQ_{A^{\op}}\Rprimetau_A (M^{\bgm}).
    \]
    So, we also have $\RQ_{A^{\op}}\Rtau_A (M^{\bgm}) $ is acyclic.
    Finally, because we have the triangle
    \[
    \RprimeQ_{A^{\op}}(\Rprimetau_{A}(M^{\bgm})) \rightarrow \RprimeQ_{A}(M^{\bgm}) \rightarrow \RprimeQ_{A^{\op}}(\RprimeQ_A(M^{\bgm})),
    \]
    we get the desired quasi-isomorphism.
\end{proof}

The following lemma gives a formula for the kernel of the composition of Fourier-Mukai functors.

\begin{lem}
    \label{lem:kernel of compositions of fmt}
    Let $A,B,C$ be noetherian locally finite $\Nbb$-graded $k$-algebras.
    We assume that $A,B,C$ have balanced dualizing complexes.
    Let $\qgrobjname{E} = \pi_{A^{\op} \otimes_k B}(E) \in D(\QGr(A^{\op} \otimes_k B))$ and $\qgrobjname{F} = \pi_{B^{\op} \otimes_k C}(F)\in D(\QGr(B^{\op} \otimes_k C))$, where $E,F$ are objects in $D(\Gr(A^{\op} \otimes_k B))$ and $D(\Gr(B^{\op} \otimes_k C))$, respectively.
    We set 
    \[
    \qgrobjname{G} := \pi_{A^{\op} \otimes_k C}(\Romega_{A^{\op} \otimes_k B}(\qgrobjname{E}) \otimes_{\Bcal}^{\LD} \Romega_{B^{\op} \otimes_k C}(\qgrobjname{F})) \in D(\QBiGr(A^{\op} \otimes_k C)).
    \]
    Then, we have a natural isomorphism
    \begin{align*}
        \Phi_{\qgrobjname{F}} \circ \Phi_{\qgrobjname{E}} \simeq \Phi_{\qgrobjname{G}}.
    \end{align*}
\end{lem}

\begin{proof}
    \label{proof:kernel of compositions of fmt}
    We have the claim by the following calculation:
\begin{align*}
    (\Phi_{\qgrobjname{F}} \circ \Phi_{\qgrobjname{E}})(-) 
    &\simeq \pi_C(\RQ_B(\Romega_A(-) \otimes_{\Acal}^{\LD} \Romega_{A^{\op} \otimes_k B}(\qgrobjname{E})) \otimes_{\Bcal}^{\LD} \Romega_{B^{\op} \otimes_k C}(\qgrobjname{F})) \\
    &\simeq \pi_C((\Romega_A(-) \otimes_{\Acal}^{\LD} \RprimeQ_B\Romega_{A^{\op} \otimes_k B}(\qgrobjname{E})) \otimes_{\Bcal}^{\LD} \Romega_{B^{\op} \otimes_k C}(\qgrobjname{F}))   \\
    &\simeq \pi_C((\Romega_A(-) \otimes_{\Acal}^{\LD} \RprimeQ_B\RprimeQ_B\RprimeQ_A(E)) \otimes_{\Bcal}^{\LD} \Romega_{B^{\op} \otimes_k C}(\qgrobjname{F})) \\
    &\simeq \pi_C((\Romega_A(-) \otimes_{\Acal}^{\LD} \RprimeQ_B\RprimeQ_A(E)) \otimes_{\Bcal}^{\LD} \Romega_{B^{\op} \otimes_k C}(\qgrobjname{F}))  \\
    &\simeq \pi_C((\Romega_A(-) \otimes_{\Acal}^{\LD} \Romega_{A^{\op} \otimes_k B}(\qgrobjname{E})) \otimes_{\Bcal}^{\LD} \Romega_{B^{\op} \otimes_k C}(\qgrobjname{F})) \\
    &\simeq \pi_C(\Romega_A(-) \otimes_{\Acal}^{\LD} (\Romega_{A^{\op} \otimes_k B}(\qgrobjname{E}) \otimes_{\Bcal}^{\LD} \Romega_{B^{\op} \otimes_k C}(\qgrobjname{F}))) \\
    &\simeq \pi_C(\Romega_A(-) \otimes_{\Acal}^{\LD} \Romega_{A^{\op} \otimes_k C}(\pi_{A^{\op} \otimes_k C}(\Romega_{A^{\op} \otimes_k B}(\qgrobjname{E}) \otimes_{\Bcal}^{\LD} \Romega_{B^{\op} \otimes_k C}(\qgrobjname{F})))) \\
    &\simeq \Phi_{\qgrobjname{G}}(-).
\end{align*}
The second isomorphism is obtained by using \cref{lem:bimodule proj formula}.
The third and fifth isomorphisms are obtained by using (3) of \cref{lem:basic properties}.
The fourth isomorphism is obtained by using \cref{lem:bimodule idempotent formula}.
The sixth isomorphism is obtained by using the associativity of tensor products.
The seventh isomorphism is obtained by using (3) of \cref {lem:basic properties}, \cref{lem:bimodule idempotent formula} and \cref{lem:bimodule proj formula} repeatedly.
\end{proof}


\begin{lem}
    \label{lem:applying re-main2}
    Let $A$ be a noetherian locally finite $\Nbb$-graded $k$-algebra.
    We assume that $A$ has a balanced dualizing complex, $\qgr (A)$ has a canonical bimodule $\canmod_A$ and the globel dimension $n$ of $\qgr (A)$ is larger than $0$.
    Then, $\R^1\torfunct{A}(A)$ is a finite $A$-module.
\end{lem}

\begin{proof}
    \label{proof:applying re-main2}
    We calculate $H^0(\qgr (A), \Ocal_A(m))$ as follows:
    \begin{align*}
        H^0(\qgr (A), \Ocal_A(m)) &= \Hom_{\qgr (A)}(\Ocal_A, \Ocal_A(m)) \\
        &= \Hom_{\qgr (A)}(\Ocal_A(m), \canmod_A[n]) \\
        &= \Hom_{\qgr (A)}(\Ocal_A, \canmod_A(-m)[n]) \\
        &= H^n(\qgr (A),  \canmod_A(-m)),
    \end{align*}
    where we simply write $\canmod_A$ for $\Ocal_A \otimes \canmod_A$.
    So, from the noncommutative Serre vanishing theorem \cite[Theorem 7.4]{artin1994noncommutative}, we have $ H^0(\qgr (A), \Ocal_A(m))=0$ for $m \ll 0$.
    This induces that $\R^1\torfunct{A}(A)$ is a finite $A$-module from \eqref{eq:canonical exact sequence}.
\end{proof}

\begin{rmk}
    From this lemma, we can apply \cref{thm:re-main2} (and \cref{prop:re-main2}) in our setting.
\end{rmk}

\begin{lem}
    \label{lem:fm-kernel of tensor products}
    Let $A$ be a locally finite noetherian $\Nbb$-graded $k$-algebra with a balanced dualizing complex. 
    Let $M \in D(\Gr(A^{\en}))$.
    Then, the functor 
    \begin{align*}
      \pi_{A}(- \otimes^{\LD}_{A} M) : D(\QGr A) \rightarrow D(\QGr A),\quad  \pi_A(N) \mapsto  \pi_{A}(N \otimes^{\LD}_{A} M)
    \end{align*}
    is naturally isomorphic to $\Phi_{\pi_{A^{\en}}(M^{\bgm})}$.
\end{lem}

\begin{proof}
    \label{proof:fm-kernel of tensor products}
    We show that 
    \begin{align*}
        \pi_A(\RQ_A(N) \otimes^{\LD}_{\dgcatname{A}} \RQ_{A^{en}}(M^{\bgm})) \simeq \pi_A(N \otimes^{\LD}_{A} M).
    \end{align*}
    Firstly, because
    \[
        \Rtau_A(N) \otimes^{\LD}_{\dgcatname{A}} \RQ_{A^{en}}(M^{\bgm}) \simeq \Rtau_A(N) \otimes_{\dgcatname{A}}^{\LD} \RprimeQ_{A^{\op}}\RprimeQ_{A}(M^{\bgm}) 
    \]
    is acyclic from the locally finite version of \cite[Proposition 4.6]{ballard2021kernels}, we have 
    \begin{align*}
        \RQ_A(N) \otimes_{\dgcatname{A}}^{\LD} \RQ_{A^{en}}(M^{\bgm}) &\simeq N \otimes_{\dgcatname{A}}^{\LD} \RQ_{A^{en}}(M^{\bgm}) \\
        &\simeq N \otimes_{\dgcatname{A}}^{\LD} \RprimeQ_A(M^{\bgm}) \\
        &\simeq \RQ_A(N \otimes_{\dgcatname{A}}^{\LD} M^{\bgm}).
    \end{align*}
    Note that the second isomorphism is obtained by using \cref{lem:a good formula fo fin mods} and the third isomorphism is obtained by using \cref{lem:bimodule proj formula}.
    So, we show that
    \[
    \pi_A(N \otimes_{\dgcatname{A}}^{\LD} M^{\bgm}) \simeq \pi_A(N \otimes_{A}^{\LD} M).
    \]
    Moreover, since $\Gr(A)$ is equivalent to the category of right modules of the $\Zbb$-algebra $\dgcatname{A}$,
    it is enough to show that $N \otimes_{\dgcatname{A}} M^{\bgm}$ and $N \otimes_{A} M$ are isomorphic as modules of the $\Zbb$-algebra $\dgcatname{A}$ for any $N \in \Gr(A)$ and $M \in \Gr(A^{\en})$.
    We can easily check this from the definition of the tensor product of modules of a $\Zbb$-algebra. 
\end{proof}


  

The following proposition is a generalization of \cite[Corollary 5.21]{huybrechts2006fourier} in noncommutative projective geometry.

\begin{prop}
    \label{prop:D-eq and dimension}
    Let $A,B$ be noetherian locally finite $\Nbb$-graded $k$-algebras.
    We assume that $A,B$ have balanced dualizing complexes.
    We assume that $\qgr (A), \qgr (B)$ have canonical bimodules $\canmod_A, \canmod_B$, respectively.

    Then, 
    \[
    D^b(\qgr (A)) \simeq D^b(\qgr (B)) \Rightarrow \gldim(\qgr (A)) = \gldim(\qgr (B)).
    \]
\end{prop}

\begin{proof}
    \label{proof:D-eq and dimension}
    For simplicity, we set $n= \gldim(\qgr (A))$ and $m= \gldim(\qgr (B))$.
    Let $F:D^b(\qgr (A)) \rightarrow D^b(\qgr (B))$ be an equivalence.
    Then, from (3) of \cref{thm:re-main2}, we have $F \simeq \Phi_{\qgrobjname{F}}$ for some $\qgrobjname{F} \in D(\QGr(A^{\op} \otimes_k B))$ and $\qgrobjname{F}$ is unique up to quasi-isomorphism.
    Let $G:D^b(\qgr (B)) \rightarrow D^b(\qgr (A))$ be the quasi-inverse of $F$.
    Then, $G \simeq \Phi_{\qgrobjname{G}}$ for some $\qgrobjname{G} \in D(\QGr(B^{\op} \otimes_k A))$ and $\qgrobjname{G}$ is unique up to quasi-isomorphism.
    Note that when $n = 0$ or $m = 0$, (3) of \cref{thm:re-main2} cannot be applied to at least one of $F$ and $G$.
     However, because we have an equivalence between $D^b_{\dg}(\qgr (A)$) and $D^b_{\dg}(\qgr (B))$ from the uniqueness of dg enhancements and obtain an equivalence between $D_{\dg}(\QGr (A)) $ and $D_{\dg}(\QGr (B))$ from \eqref{eq:semi-free quasi-equivalence}, so we can also take  Fourier-Mukai functors as an equivalence $F: D^b(\qgr (A)) \rightarrow D^b(\qgr (B))$ and a quasi-inverse $G$ of $F$ (cf. \cite[Corollary 4.18]{ballard2021kernels}).
    Then, $G$ is a left adjoint of $F$.
    Moreover, $H:=\serrefunct_{D^b(\qgr (A))} \circ G \circ \serrefunct_{D^b(\qgr (B))}^{-1}$ is a right adjoint of $F$ (\cite[Remark 3.31]{huybrechts2006fourier}), where $\serrefunct_{D^b(\qgr (A))}$ (resp. $\serrefunct_{D^b(\qgr (B))}$) is the Serre functor of $D^b(\qgr (A))$ (resp. $D^b(\qgr (B))$).
    Thus, we have $G \simeq H$.
    In addtition, $\serrefunct_{D^b(\qgr (A))}$ and $\serrefunct_{D^b(\qgr (B))}$ are given by $- \otimes_A \pi_A\left(H^{-(n+1)}(R_A)\right)[n]$ and $- \otimes_B \pi_B\left(H^{-(m+1)}(R_B)\right)[m]$, respectively by our assumption (in detail, see \cref{rmk:D-eq and dimension}).
    So, from \cref{lem:fm-kernel of tensor products}, we have
    \begin{align*}
        \serrefunct_{D^b(\qgr (A))} &\simeq \Phi_{\pi_{A^{\en}}\left(H^{-(n+1)}(R_A^{\bgm})\right)}[n], \\
        \serrefunct_{D^b(\qgr (B))} &\simeq \Phi_{\pi_{B^{\en}}\left(H^{-(m+1)}(R_B^{\bgm})\right)}[m],
    \end{align*}
    where $R_A^{\bgm}$ and $R_B^{\bgm}$ are the associated objects of $R_A$ and $R_B$ in $D(\BiGr(A \otimes_k A^{\op}))$ and $D(\BiGr(B \otimes_k B^{\op}))$, respectively.
    From \cref{lem:kernel of compositions of fmt}, the kernel $\qgrobjname{E}''$ of $G \circ \serrefunct_{D^b(\qgr (B))}^{-1}$ is given by 
    \[
    \qgrobjname{E}'' \simeq \pi_{A \otimes_k B^{\op}}( \RQ_{B^{\op} \otimes_k B}(H^{-(m+1)}(R_B^{\bgm})^{\vee})\otimes_{\Bcal}^{\LD}\Romega_{A \otimes_k B^{\op}}(\qgrobjname{E}') )[-m],
    \]
    where $(-)^{\vee} = \Hom_{B}(-,B)$ (in detail, see \cref{rmk:D-eq and dimension}).
    By using \cref{lem:kernel of compositions of fmt} again, the kernel $\qgrobjname{E}'''$ of $H$ is given by
    \[
    \qgrobjname{E}''' \simeq \pi_{A \otimes_k B^{\op}}(\Romega_{A \otimes_k B^{\op}}(\qgrobjname{E}'') \otimes_{\Acal}^{\LD} \RQ_{A^{\op} \otimes_k A}(H^{-(n+1)}(R_A^{\bgm})))[n].
    \]
    For simplicity, we put $M_A = H^{-(n+1)}(R_A)$ and $M_B = H^{-(m+1)}(R_B)$.
    By using \cref{lem:basic properties}, \cref{lem:bimodule idempotent formula}, \cref{lem:bimodule proj formula} and \cref{lem:a good formula fo fin mods} repeatedly, we arrange $\qgrobjname{E}'''$ as follows:
    \begin{align*}
        \qgrobjname{E}''' 
        &\simeq \scalebox{0.8}{$
            \pi_{A \otimes_k B^{\op}}
                (\RQ_{A \otimes_k B^{\op}} (\RQ_{B^{\op} \otimes_k B}(M_B^{\bgm {\vee}}) \otimes_{\Bcal}^{\LD} 
                 \Romega_{A \otimes_k B^{\op}}(\qgrobjname{E}')) \otimes_{\Acal}^{\LD} \RQ_{A^{\op} \otimes_k A}(M_A^{\bgm}))
            [n-m]
         $}
          \\
        &\simeq \scalebox{0.8}{$
            \pi_{A \otimes_k B^{\op}}
                ((\RprimeQ_{B^{\op}} \circ \RQ_{B^{\op}\otimes_k B})(M_B^{\bgm {\vee}}) \otimes_{\Bcal}^{\LD}
                (\RprimeQ_{A} \circ \Romega_{B^{\op} \otimes_k A})(\qgrobjname{E}') 
        \otimes_{\Acal}^{\LD} \RQ_{A^{\op} \otimes_k A}(M_A^{\bgm}))
        [n-m]
        $}
        \\
        &\simeq \pi_{A \otimes_k B^{\op}}(\RQ_{B^{\op} \otimes_k B}(M_B^{\bgm {\vee}}) \otimes_{\Bcal}^{\LD} \Romega_{B^{\op} \otimes_k A}(\qgrobjname{E}') \otimes_{\Acal}^{\LD} \RQ_{A^{\op} \otimes_k A}(M_A^{\bgm}))[n-m] \\
        &\simeq \pi_{A \otimes_k B^{\op}}(\RprimeQ_{B^{\op}}(M_B^{\bgm \vee}) \otimes_{\Bcal}^{\LD} \Romega_{B^{\op} \otimes_k A}(\qgrobjname{E}') \otimes_{\Acal}^{\LD} \RprimeQ_{A}(M_A^{\bgm}))[n-m] \\
        &\simeq \pi_{A \otimes_k B^{\op}}(\RprimeQ_{B^{\op}}(M_B^{\bgm \vee} \otimes_{\Bcal}^{\LD} \Romega_{B^{\op} \otimes_k A}(\qgrobjname{E}') \otimes_{\Acal}^{\LD} \RprimeQ_{A}(M_A^{\bgm})))[n-m] \\
        &\simeq \pi_{A \otimes_k B^{\op}}((\RprimeQ_{B^{\op}} \circ \RprimeQ_{A})(M_B^{\bgm \vee} \otimes_{\Bcal}^{\LD} \Romega_{B^{\op} \otimes_k A}(\qgrobjname{E}') \otimes_{\Acal}^{\LD} \RprimeQ_{A}(M_A^{\bgm})))[n-m] \\
        &\simeq (\pi_{A \otimes_k B^{\op}} \circ \RQ_{A \otimes_k B^{\op}})(M_B^{\bgm \vee} \otimes_{\Bcal}^{\LD} \Romega_{B^{\op} \otimes_k A}(\qgrobjname{E}') \otimes_{\Acal}^{\LD} M_A^{\bgm})[n-m] \\
        &\simeq \pi_{A \otimes_k B^{\op}}(M_B^{\bgm \vee} \otimes_{\Bcal}^{\LD} \Romega_{B^{\op} \otimes_k A}(\qgrobjname{E}') \otimes_{\Acal}^{\LD} M_A^{\bgm})[n-m] \\
        &\simeq \pi_{A \otimes_k B^{\op}}(M_B^{\bgm \vee} \otimes_{\Bcal} \Romega_{B^{\op} \otimes_k A}(\qgrobjname{E}') \otimes_{\Acal} M_A^{\bgm})[n-m].
    \end{align*}
    From (3) of \cref{thm:re-main2} (see also \cref{rmk:re-main2}), we have an quasi-isomorphism
    \[
    \pi_{A \otimes_k B^{\op}}(\Romega_{B^{\op} \otimes_k A}(\qgrobjname{E}')) \simeq \pi_{A \otimes_k B^{\op}}(M_B^{\bgm \vee} \otimes_{\Bcal} \Romega_{B^{\op} \otimes_k A}(\qgrobjname{E}') \otimes_{\Acal} M_A^{\bgm})[n-m]
    \]
    in $D(\QGr(A \otimes_k B^{\op}))$.
    This induces $n=m$.
\end{proof}

\begin{rmk}
    \label{rmk:D-eq and dimension}
    In this remark, we give a detailed description about the Serre functor of $D^b(\qgr (A))$ when $A$ is a locally finite noetherian $\Nbb$-graded $k$-algebra with a balanced dualizing complex $R_A$ and $\qgr (A)$ has a canonical bimodule $\canmod_A$.

    Firstly, we recall that the Serre functor $\Scal_{D^b(\qgr (A))}$ of $D^b(\qgr (A))$ is given by 
    \[
     \Scal_{D^b(\qgr (A))}(\pi_A(M)) \simeq \pi_A(M \otimes^{\LD}_A R_A)[-1] \simeq \pi_A(M) \otimes \canmod_A[n],
    \]
    where $n$ is the global dimension of $\qgr (A)$.
    In particular, because we have $\pi_A(N \otimes^{\LD}_A R_A) \simeq \pi_A(N) \otimes \canmod_A[n+1] \in \qgr(A)[n+1]$ for all $N \in \gr(A)$, we have $\pi_A(R_A) \simeq \pi_A(M_A)[n+1] \in D^b(\qgr (A))$ and 
    \[
    \Scal_{D^b(\qgr (A))}(\pi_A(N)) \simeq \pi_A(N \otimes^{\LD}_A M_A)[n] \simeq \pi_A(N \otimes_A M_A)[n],
    \]
    where we put $M_A =H^{-(n+1)}(R_A)$ for simplicity.
    Note that from these isomorphisms, we obtain the well-definedness of $\pi_A(- \otimes_A M_A)[n]$ as a endofunctor of $D^b(\qgr (A))$.
    
    On the other hand, the inverse of the Serre functor $\Scal_{D^b(\qgr (A))}$ is given by
    \[
     \Scal_{D^b(\qgr (A))}^{-1}(\pi_A(N)) \simeq \pi_A(\RHom_A(R_A,\pi_{A}(N)))[1]
    \]
    from \cite[Proposition A.3]{de2004ideal}.
    Then, we have a natural morphism
    \[
     \pi_A(N \otimes_A^{\LD} \RHom_A(R_A,A)) \rightarrow \pi_A(\RHom_A(R_A,N)).
    \]
    This morphism is an isomorphism because $D^b(\qgr(A))$ is classically generated by $\{\Ocal_A(m)\}_{m \in \Zbb}$ (\cite[Lemma 4.2.2]{bondal2002generators}) and we can apply \cite[Propostion 5.3.22]{yekutieli2019derived}.
    So, we have 
    \begin{align*}
     \Scal_{D^b(\qgr (A))}^{-1}(\pi_A(N)) &\simeq  \pi_A(N \otimes_A^{\LD} \RHom_A(R_A,A))[1] \\
      &\simeq \pi_A(N \otimes_A^{\LD} \RHom_A(M_A, A))[1] \\
      &\simeq \pi_A(N \otimes^{\LD}_A M_A^{\vee})[1] \\
     &\simeq \pi_A(N \otimes_A M_A^{\vee})[1],
    \end{align*}
    where $M_A^{\vee} = \Hom_A(M_A,A)$.
    In addition, we have the following isomorphisms in $D(\QGr(A^{\en}))$:
    \begin{align*}
        \pi_{A^{\en}}(R_A \otimes_A^{\LD} \RHom_A(R_A,A)) &\simeq \pi_{A^{\en}}(M_A \otimes_A M_A^{\vee}), \\  
        \pi_{A^{\en}}(\RHom_A(R_A,R_A) \otimes_A^{\LD} R_A) &\simeq \pi_{A^{\en}}(M_A^{\vee} \otimes_A M_A),
    \end{align*}
    which are used in the last isomorphism of the arrangement of $\qgrobjname{E}'''$ in the proof of \cref{prop:D-eq and dimension}.
\end{rmk}

\begin{dfn}
    \label{def:canring}
    Let $A$ be a noetherian locally finite $\Nbb$-graded $k$-algebra.
    We assume that $\qgr (A)$ has a canonical bimodule $\canmod_A$.
    We define the \emph{canonical graded $k$-algebra} $\canring(\qgr (A))$ of $A$ by $\canring(\qgr (A)) = B(\qgr (A), \Ocal_A, - \otimes \canmod_A)$.
    We define the \emph{anti-canonical graded $k$-algebra} $\anticanring(\qgr (A))$, similarly.
\end{dfn}

The following proposition is a generalization of \cite[proposition 6.1]{huybrechts2006fourier} in the noncommutative setting.
\begin{prop}
    \label{prop:re-main1}
    Let $A,B$ be noetherian locally finite $\Nbb$-graded $k$-algebras.
    We assume that $A,B$ have balanced dualizing complexes.
    We assume that $\qgr (A), \qgr (B)$ have canonical bimodules $\canmod_A, \canmod_B$.
    Then,
        \[
         D^b(\qgr (A)) \simeq D^b(\qgr (B)) \ \Longrightarrow 
         \begin{array}{l}
            \canring(\qgr (A)) \simeq \canring(\qgr (B) ), \\
            \anticanring(\qgr (A)) \simeq \anticanring(\qgr (B))
         \end{array}
        \]
        as graded $k$-algebras.
\end{prop}

\begin{proof}
    \label{proof:prop-main1}
    We only show the isomorphism of canonical graded $k$-algebras.

    \noindent
    \emph{Step 1.}
    Let $F:D^b(\qgr (A)) \rightarrow D^b(\qgr (B))$ be an equivalence.
    Then, from (3) of \cref{thm:re-main2}, we have $F \simeq \Phi_{\qgrobjname{F}}$ for some $\qgrobjname{F} \in D(\QGr(A^{\op} \otimes_k B))$ and $\qgrobjname{F}$ is unique up to quasi-isomorphism.
    Let $G:D^b(\qgr (B)) \rightarrow D^b(\qgr (A))$ be the quasi-inverse of $F$.
    Then, $G \simeq \Phi_{\qgrobjname{G}}$ for some $\qgrobjname{G} \in D(\QGr(B^{\op} \otimes_k A))$ and $\qgrobjname{G}$ is unique up to quasi-isomorphism.
    Note that when $n = 0$ or $m = 0$, (3) of \cref{thm:re-main2} cannot be applied to at least one of $F$ and $G$.
    However, because we have an equivalence between $D^b_{\dg}(\qgr (A)$) and $D^b_{\dg}(\qgr (B))$ from the uniqueness of dg enhancements and obtain an equivalence between $D_{\dg}(\QGr (A)) $ and $D_{\dg}(\QGr (B))$ from \eqref{eq:semi-free quasi-equivalence}, so we can also take  Fourier-Mukai functors as an equivalence $F: D^b(\qgr (A)) \rightarrow D^b(\qgr (B))$ and a quasi-inverse $G$ of $F$ (cf.  \cite[Corollary 4.18]{ballard2021kernels}).

    Because $A^{\op} \otimes_k B \simeq A^{\op} \otimes_k (B^{\op})^{\op}$ and $B^{\op} \otimes_k A \simeq B^{\op} \otimes_k (A^{\op})^{\op}$, $\qgrobjname{F}, \qgrobjname{G}$ determine the functors 
    \begin{align*}
        \Phi'_{\qgrobjname{F}}:D^b(\qgr (B^{\op})) \rightarrow D^b(\qgr (A^{\op})),\\
         \Phi'_{\qgrobjname{G}}:D^b(\qgr (A^{\op})) \rightarrow D^b(\qgr (B^{\op})).
    \end{align*}
    We show that $\Phi'_{\qgrobjname{F}}$ and $\Phi'_{\qgrobjname{G}}$ are equivalences.
    We consider the kernel $\qgrobjname{E}$ of the functor $\Phi'_{\qgrobjname{F}} \circ \Phi'_{\qgrobjname{G}}$.
    By using \cref{lem:kernel of compositions of fmt}, we have
    \begin{align*}
        \qgrobjname{E} = \pi_{A^{\op} \otimes_k A}(\Romega_{B^{\op} \otimes_k (A^{\op})^{\op}}(\qgrobjname{G}) \otimes_{\Bcal^{\op}}^{\LD} \Romega_{A^{\op} \otimes_k (B^{\op})^{\op}}(\qgrobjname{F})) \in D(\QGr(A^{\op} \otimes_k A)).
    \end{align*}
    So, to prove that $\Phi'_{\qgrobjname{F}} \circ \Phi'_{\qgrobjname{G}}$ is the identity functor, it is enough to show that
    \begin{equation}
        \label{eq: claim of step 1}
    \qgrobjname{E} \simeq \pi_{A^{\op} \otimes_k A}(\Romega_{A^{\op} \otimes_k B}(\qgrobjname{F}) \otimes_{\Bcal}^{\LD} \Romega_{B^{\op} \otimes_k A}(\qgrobjname{G})) \tag{$\ast$}.
    \end{equation}
    because the right hand side of this isomorphism is quasi-isomorphic to $\pi_{A^{\op} \otimes_k A}(A^{\bgm})$ by the uniqueness of Fourier-Mukai kernels.
    The isomorphism $\eqref{eq: claim of step 1}$ follows from the following sub-lemma.
    \begin{sublem}
    \label{sublem:tensor product of dg bimodules}
    Let $\Acal, \Bcal$ and $\Ccal$ be dg categories.
    Let $M$ be a dg $\Acal$-$\Bcal$-bimodule and $N$ be a dg $\Bcal$-$\Ccal$-bimodule.
    Then, we have a natural isomorphism of dg $\Acal$-$\Ccal$-bimodules
    \begin{align*}
        M \otimes_{\Bcal} N \simeq N \otimes_{\Bcal^{\op}} M.
    \end{align*}
    \end{sublem}
    \begin{proof}[Proof of \cref{sublem:tensor product of dg bimodules}]
        \label{proof:sublem-tensor product of dg bimodules}
        Note that for any $a \in \Acal, c \in \Ccal$, $(M \otimes_{\Bcal} N)(a,c)$ is given by 
        \[
        (M \otimes_{\Bcal} N)(a,c) = \int^{b \in \Bcal} M(a,b) \otimes_k  N(b,c).
        \] 
        So, the claim is proved by the following coend calculation:
        \begin{align*}
            (M \otimes_{\Bcal} N)(a,c) 
            &= \int^{b \in \Bcal} M(a,b) \otimes_k  N(b,c) \\
            &\simeq \int^{b \in \Bcal^{\op}} N(b,c) \otimes_k  M(a,b) 
            = (N \otimes_{\Bcal^{\op}} M)(a,c).
        \end{align*}
    \end{proof}
    In the same way, we can show that $\Phi'_{\qgrobjname{G}} \circ \Phi'_{\qgrobjname{F}}$ is isomorphic to the identity functor.
    Thus, $\Phi'_{\qgrobjname{F}}$ and $\Phi'_{\qgrobjname{G}}$ are equivalences.

    \noindent
    \emph{Step 2.}
    Let $\qgrobjname{H}:= \Romega_{A^{\op} \otimes_k B}(\qgrobjname{F}) \otimes_k \Romega_{(A^{\op})^{\op} \otimes_k B^{\op}}(\qgrobjname{G})$.
    Then, $\qgrobjname{H}$ defines the functor 
    \begin{align*}
        \tilde{\Phi}_{\qgrobjname{H}}:D(\QBiGr (A \otimes_k A^{\op})) \rightarrow D(\QBiGr(B \otimes_k B^{\op})), \\
        \tilde{\Phi}_{\qgrobjname{H}}(-) = \pi_{B \otimes_k B^{\op}}(\Romega_{A \otimes_k A^{\op}}(-) \otimes_{\dgcatname{A \otimes A^{\op}}}^{\LD} \qgrobjname{H}).
    \end{align*}
        This functor is an equivalence since $\Phi_{\qgrobjname{F}}$ and $\Phi'_{\qgrobjname{G}}$ are equivalences (in detail, see \cref{rmk:prop-main1}).


        As in the proof of \cref{prop:D-eq and dimension}, we put $M_A = H^{-(n+1)}(R_A)$ and $M_B = H^{-(m+1)}(R_B)$, where $n,m$ are the global dimensions of $\qgr (A), \qgr (B)$.
        We also put $\qgrobjname{M}_l= \tilde{\Phi}_{\qgrobjname{H}}(\pi_{A \otimes_k A^{\op}}(M_A^{\otimes l \bgm}))$.
        Then, we show that 
        \[
        \Phi_{\qgrobjname{M}_l}:D(\QGr (B)) \rightarrow D(\QGr (B))
        \]
        induces the functor
        \[
        \Phi_{\qgrobjname{M}_l}:D^b(\qgr (B)) \rightarrow D^b(\qgr (B)).
        \]
        and $\Phi_{\qgrobjname{M}_l}$ is an equivalence.
        To see this, we prove that $\qgrobjname{M}_l$ is quasi-isomorphic to the kernel of the composition of the following functors:
    \[
    \begin{tikzcd}
         D^b(\qgr (B)) \arrow[r, "\Phi_{\qgrobjname{G}}"] & D^b(\qgr (A)) \arrow[r, "\Phi_{\pi_{A \otimes_k A^{\op}}\left(M_A^{\otimes l \bgm}\right)}"]&[6em] D^b(\qgr (A)) \arrow[r, "\Phi_{\qgrobjname{F}}"] & D^b(\qgr (B)).
    \end{tikzcd}
    \]
    Here, we need the following sub-lemma. 
    \begin{sublem}
        \label{sublem:tensor product of dg bimodules 2}
        Let $\Acal, \Bcal$ be dg categories.
        Let $\dgmodname{L}$ be a dg $\Acal$-$\Acal$-bimodule, $\dgmodname{M}$ be a dg $\Acal$-$\Bcal$-bimodule and $\dgmodname{N}$ be a dg $\Acal^{\op}$-$\Bcal^{\op}$-bimodule.
        Then, we have a natural isomorphism of dg $\Bcal$-$\Bcal$-bimodules
        \begin{align*}
            \dgmodname{L} \otimes_{\Acal^{\en}} (\dgmodname{M} \otimes_k \dgmodname{N}) \simeq \dgmodname{N} \otimes_{\Acal}  (\dgmodname{L} \otimes_{\Acal} \dgmodname{M}).
        \end{align*}
    \end{sublem}
    \begin{proof}[Proof of \cref{sublem:tensor product of dg bimodules 2}]
        This is proved by the following coend calculation:
    \begin{align*}
       ( \dgmodname{L} \otimes_{\Acal^{\en}} (\dgmodname{M} \otimes_k \dgmodname{N}))(b,b') 
       &= \int^{(a,a') \in \dgcatname{A \otimes A^{\op}}} \dgmodname{L}(a',a) \otimes_k (\dgmodname{M}(a,b) \otimes_k \dgmodname{N}(a',b')) \\
       &\simeq \int^{(a,a') \in \dgcatname{A \otimes A}} \dgmodname{N}(a',b') \otimes_k (\dgmodname{L}(a',a) \otimes_k \dgmodname{M}(a,b)) \\
       &\simeq \int^{a' \in \Acal} \int^{a \in \Acal} \dgmodname{N}(a',b') \otimes_k (\dgmodname{L}(a',a) \otimes_k \dgmodname{M}(a,b)) \\
       &\simeq \int^{a' \in \Acal} \dgmodname{N}(a',b') \otimes_k \int^{a \in \Acal} \dgmodname{L}(a',a) \otimes_k \dgmodname{M}(a,b) \\
       &\simeq \int^{a' \in \Acal} \dgmodname{N}(a',b') \otimes_k (L \otimes_{\Acal} M)(a',b) \\
       &= (\dgmodname{N} \otimes_{\Acal}  (\dgmodname{L} \otimes_{\Acal} \dgmodname{M}))(b,b').
    \end{align*}
     At the third equality, we use Fubini theorem (\cite[Proposition 3.12]{genovese2017adjunctions}, \cite[Proposition 2.17]{imamura2024formal}).
     At the fourth equality, we use \cite[Proposition 2.15]{imamura2024formal}.
    \end{proof}
    We see that the kernel of $\Phi_{\qgrobjname{F}} \circ \Phi_{\pi_{A \otimes_k A^{\op}}\left(M_A^{\otimes l \bgm}\right)} \circ \Phi_{\qgrobjname{G}}$ is isomorphic to $\qgrobjname{M}_l$.
    Actually, from \cref{lem:bimodule idempotent formula}, \cref{lem:bimodule proj formula}, \cref{lem:kernel of compositions of fmt} and \cref{sublem:tensor product of dg bimodules 2}, we have 
    \begin{align*}
        &\left(\text{the kernel of} \ \Phi_{\qgrobjname{F}} \circ \Phi_{\pi_{A \otimes_k A^{\op}} \left(M_A^{\otimes l \bgm}\right)} \circ \Phi_{\qgrobjname{G}}\right)\\
        &\simeq \pi_{B^{\en}}(\RQ_{B^{\op} \otimes_k A}(\Romega_{B^{\op} \otimes_k A}(\qgrobjname{G}) \otimes_{\Acal}^{\LD} \RQ_{A^{\en}}(M_A^{\otimes l \bgm})) \otimes_{\Acal}^{\LD} \Romega_{A^{\op} \otimes_k B}(\qgrobjname{F})) \\
        &\simeq \pi_{B^{\en}}(\Romega_{B^{\op} \otimes_k A}(\qgrobjname{G}) \otimes_{\Acal}^{\LD} \RQ_{A^{\en}}(M_A^{\otimes l \bgm}) \otimes_{\Acal}^{\LD} \Romega_{A^{\op} \otimes_k B}(\qgrobjname{F})) \\
        &\simeq\qgrobjname{M}_{l}.
    \end{align*}
   Note that $\Phi_{\pi_{A \otimes_k A^{\op}}\left(M_A^{\otimes l \bgm}\right)} \simeq \Scal_A^l[-ln]$, where $\Scal_A$ is the Serre functor of $D^b(\qgr (A))$.
   So, since Serre functors commute with equivalences, we have $\Phi_{\qgrobjname{M}_l} \simeq \Scal_B[-ln]$.
   Moreover, by \cref{prop:D-eq and dimension} and the uniqueness of Fourier-Mukai kernels ((3) of \cref{thm:re-main2}), we have $\qgrobjname{M}_l \simeq \pi_{B \otimes_k B^{\op}}(M_B^{\otimes l \bgm})$.
   Thus, for all $l \in \Zbb$, we have $\tilde{\Phi}_{\qgrobjname{H}}(\pi_{A \otimes_k A^{\op}}(M_A^{\otimes l \bgm})) \simeq \pi_{B \otimes_k B^{\op}}(M_B^{\otimes l \bgm})$.

   \noindent
   \emph{Step 3.}
   Since $\tilde{\Phi}_{\qgrobjname{H}}$ is an equivalence, we have 
   \begin{align*}
    &\Hom_{D(\QBiGr (A \otimes_k A^{\op}))}(\pi_{A \otimes_k A^{\op}}(M_A^{\otimes l_1 \bgm}), \pi_{A \otimes_k A^{\op}}(M_A^{\otimes l_2 \bgm})) \\
    &\simeq \Hom_{D(\QBiGr (B \otimes_k B^{\op}))}(\pi_{B \otimes_k B^{\op}}(M_B^{\otimes l_1 \bgm}), \pi_{B \otimes_k B^{\op}}(M_B^{\otimes l_2 \bgm}))
   \end{align*}
    for all $l_1, l_2 \in \Zbb$.
   In particular, we have
   \begin{align*}
    &H^0(\qgr (A), \pi_A(M_A^{\otimes l}))\\
    &\simeq \Hom_{D(\QBiGr (A \otimes_k A^{\op}))}(\pi_{A \otimes_k A^{\op}}(A^{\bgm}), \pi_{A \otimes_k A^{\op}}(M_A^{\otimes l \bgm})) \\
    &\simeq\Hom_{D(\QBiGr (B \otimes_k B^{\op}))}(\pi_{B \otimes_k B^{\op}}(B^{\bgm}), \pi_{B \otimes_k B^{\op}}(M_B^{\otimes l \bgm})) \\
    &\simeq H^0(\qgr (B), \pi_B(M_B^{\otimes l}))
    \end{align*} 
    for all $l \in \Zbb$.
    Here, the first isomorphism
    is obtained as follows: we have
    \begin{align*}
        H^0(\qgr (A), \pi_A(M_A^{\otimes l})) 
        &= \Hom_{\qgr (A)}(\pi_A(A), \pi_A(M_A^{\otimes l})) \\
        &\simeq \Hom_{D(\Gr A)}(A, \RQ_A(M_A^{\otimes l})) \\
        &\simeq H^0(\RHom_{\Gr A}(A, \RQ_A(M_A^{\otimes l}))) \\
        &\simeq \Hom_{\Gr A}(A, \RQ_A(M_A^{\otimes l})) \\
        &\simeq H^0(\RQ_A(M_A^{\otimes l}))_0 \\
        &\simeq \Q_A(M_A^{\otimes l})_0
    \end{align*}
    and
    \begin{align*}
        &\Hom_{D(\QBiGr (A \otimes_k A^{\op}))}(\pi_{A \otimes_k A^{\op}}(A^{\bgm}), \pi_{A \otimes_k A^{\op}}(M_A^{\otimes l \bgm})) \\
        &\simeq \Hom_{D(\QBiGr(A \otimes_k A^{\op}))}(\pi_{A \otimes_k A^{\op}}((A^{\bgm})_{\geq 0, \geq0}), \pi_{A \otimes_k A^{\op}}(M_A^{\otimes l \bgm})) \\
        &\simeq \Hom_{D(\BiGr(A \otimes_k A^{\op}))}((A^{\bgm})_{\geq 0, \geq0}, \RQ_{A \otimes_k A^{\op}}(M_A^{\otimes l \bgm})) \\
        &\simeq H^0(\R\Hom_{\BiGr(A \otimes_k A^{\op})}((A^{\bgm})_{\geq 0, \geq0}, \RQ_{A \otimes_k A^{\op}}(M_A^{\otimes l \bgm}))) \\
        &\simeq \Hom_{\BiGr(A \otimes_k A^{\op})}((A^{\bgm})_{\geq 0, \geq0}, \RQ_{A \otimes_k A^{\op}}(M_A^{\otimes l \bgm})) \\
        &\simeq H^0(\RQ_{A \otimes_k A^{\op}}(M_A^{\otimes l \bgm}))_{0,0}\\
        &\simeq H^0(\RprimeQ_{A}(M_A^{\otimes l \bgm}))_{0,0} \\
        &\simeq \Q_A(M_A^{\otimes l})_0.
    \end{align*}
    Note that in the above calculation, we use \cref{lem:a good formula fo fin mods}, the fact that $(A^{\bgm})_{\geq 0, \geq 0}$ is a projective object in $\BiGr (A \otimes_k A^{\op})$ and an isomorphism $\pi_{A \otimes_k A^{\op}}(A^{\bgm}) \simeq \pi_{A \otimes_k A^{\op}}((A^{\bgm})_{\geq 0, \geq 0}) \in \QBiGr(A \otimes_k A^{\op})$, which is obtained from the definition of $\QBiGr(A \otimes_k A^{\op})$.

    Thus, we obtain a graded $k$-algebra isomorphism $\canring(\qgr (A)) \simeq \canring(\qgr (B))$.

\end{proof}

\begin{rmk}
    \label{rmk:prop-main1}
    We explain the reason why $\tilde{\Phi}_{\qgrobjname{H}}$ is an equivalence in this remark.
    Let $\Hcal' := \Romega_{A^{\op} \otimes_k (B^{\op})^{\op}}(\qgrobjname{G}) \otimes_k \Romega_{B^{\op} \otimes_k A}(\qgrobjname{F})$.
    We define the functor
    \begin{align*}
        \widetilde{\Phi}_{\Hcal'}:D(\QBiGr (B \otimes_k B^{\op})) \rightarrow D(\QBiGr(A \otimes_k A^{\op})), \\
        \widetilde{\Phi}_{\Hcal'}(-) = \pi_{A \otimes_k A^{\op}}(\Romega_{B \otimes_k B^{\op}}(-) \otimes_{\dgcatname{B \otimes B^{\op}}}^{\LD} \Hcal').
    \end{align*}
    Then, we calculate $\widetilde{\Phi}_{\Hcal'} \circ \widetilde{\Phi}_{\Hcal}$ 
    \begin{align*}
       \widetilde{\Phi}_{\Hcal'} \circ \widetilde{\Phi}_{\Hcal}(-) &= \pi_{A^{\en}}(\Romega_{B^{\en}} (\pi_{B^{\en}}( \Romega_{A^{\en}}(-) \otimes_{\Acal^{\en}}^{\LD} \Hcal)) \otimes_{\Bcal^{\en}}^{\LD} \Hcal') \\
        &= \pi_{A^{\en}}(\RQ_{B^{\en}}(\Romega_{A^{\en}}(-) \otimes_{\Acal^{\en}}^{\LD} \Hcal) \otimes_{\Bcal^{\en}}^{\LD} \Hcal') \\
        &= \pi_{A^{\en}}(\Romega_{A^{\en}}(-) \otimes_{\Acal^{\en}}^{\LD} \RprimeQ_{B^{\en}}(\Hcal) \otimes_{\Bcal^{\en}}^{\LD} \Hcal') \\
        &= \pi_{A^{\en}}(\Romega_{A^{\en}}(-) \otimes_{\Acal^{\en}}^{\LD} \Hcal \otimes_{\Bcal^{\en}}^{\LD} \Hcal') \\
        &= \pi_{A^{\en}}(\Romega_{A^{\en}}(-) \otimes_{\Acal^{\en}}^{\LD} \RQ_{A^{\en}}(A^{\bgm} \otimes_k A^{\bgm}) ) \\
        &= \pi_{A^{\en}}(\Romega_{A^{\en}}(-)) \\
        &= \idfun_{D(\QBiGr(A \otimes_k A^{\op}))}(-).
    \end{align*}
    In the same way, we can show that $\widetilde{\Phi}_{\Hcal} \circ \widetilde{\Phi}_{\Hcal'}$ is isomorphic to the identity functor.
    In the above calculation, we use the following formulas about functors on the derived categories of $\Zbb^4$-graded $A^{\en} \otimes_k B^{\en}$-modules, where we regard $A^{\en} \otimes_k B^{\en}$ as an $\Nbb^4$-graded algebra by the natural $\Nbb^4$-grading:
    \begin{enumerate}[label=(\arabic*)]
        \item $\RQ_{A^{\en} \otimes_k B^{en}} \simeq \RprimeQ_{A^{\en}} \circ \RprimeQ_{B^{\en}} \simeq \RprimeQ_{B^{\en}} \circ \RprimeQ_{A^{\en}}$, where functors $\RprimeQ_{A^{\en}}$, $\RprimeQ_{B^{\en}}$ and $\RQ_{A^{\en} \otimes_k B^{\en}}$ are functors from $D(\BiGr(A^{\en} \otimes_k B^{\en}))$ to $D(\BiGr(A^{\en} \otimes_k B^{en}))$.
        They are defined in the same way as in $\RprimeQ_{A}, \RprimeQ_{B}$ and $\RQ_{A \otimes_k B}$.
        We also have similar functors such as $\RprimeQ_{A \otimes_k B^{\op}}$ and $\RprimeQ_{A^{\op} \otimes_k B}$, which are endofunctors of $D(\BiGr(A \otimes_k B^{\op}))$ and $D(\BiGr(A^{\op} \otimes_k B))$, respectively.
        \item $\RQ_{A^{\en} \otimes_k B^{\en}} (M \otimes_k N) \simeq \RQ_{A^{\en}}(M) \otimes_k \RQ_{B^{\en}}(N)$ for $M \in D(\BiGr(A^{\en}))$ and $N \in D(\BiGr(B^{\en}))$.
        We also have similar formulas such as $\RprimeQ_{A^{\en} \otimes_k B^{\en}} (M \otimes_k N) \simeq \RprimeQ_{A \otimes_k B^{\op}} (M) \otimes_k \RprimeQ_{A^{\op} \otimes_k B} (N)$ for $M \in D(\BiGr(A \otimes_k B^{\op}))$ and $N \in D(\BiGr(A^{\op} \otimes_k B))$.
        \item The projection formula: $\RQ_{B^{en}}(N \otimes_{\Acal^{\en}}^{\LD} M) \simeq N \otimes_{\Acal^{\en}}^{\LD} \RprimeQ_{B^{en}}(M)$ for $M \in D(\BiGr(A^{\en} \otimes_k B^{\en}))$ and $N \in D(\BiGr(A^{\en}))$.
    \end{enumerate}
   (1) is proved in the same way as \cite[Proposition 3.41]{ballard2021kernels} (see also \cref{lem:basic properties}).
   (2) is proved by using (1) and a direct calculation.
   (3) is proved in the same way as \cref{lem:bimodule proj formula}. 
   The main reason why the similar proofs work is that $A^{\en}, B^{\en}$ and $A^{\en} \otimes_k B^{\en}$ are finitely generated as $k$-algebras and $\RQ_{A^{\en}}, \RQ_{B^{\en}}$ commute with arbitrary direct sums because $\RQ_{A^{\en}} = \RprimeQ_A \circ \RprimeQ_{A^{\op}}, \RQ_{B^{\en}} = \RprimeQ_B \circ \RprimeQ_{B^{\op}}$ and $\RQ_A, \RQ_{A^{\op}}, \RQ_B, \RQ_{B^{\op}}$ commute with arbitrary direct sums.
\end{rmk}

To prove the main theorem, we need to recall the definition of AZ-(anti-)ampleness.

\begin{dfn}[{\cite[page 250]{artin1994noncommutative}}]
    \label{def:AZ-ample}
    Let $(\Ocal, s)$ be an algebraic pair for a $k$-linear category $\abcatname{C}$.
    Then, $(\Ocal, s)$ is \emph{AZ (Artin-Zhang)-ample} if 
    \begin{enumerate}[label=(\alph*)]
        \item for every $\abobjname{M} \in \abcatname{C}$, there are $l_1, \cdots, l_p \in \Nbb$ and an epimorphism $\bigoplus_{i=1}^p s^{-l_p}(\Ocal) \rightarrow \abobjname{M}$ in $\abcatname{C}$,
        \item for every epimorphism $f:\abobjname{M} \rightarrow \Ncal$ in $\abcatname{C}$, there exists $n_0$ such that the natural map $\Hom_\abcatname{C}(s^{-n}(\Ocal), \abobjname{M}) \rightarrow \Hom_\abcatname{C}(s^{-n}(\Ocal), \Ncal) $ for all $n \geq  n_0$ is surjective.
    \end{enumerate}
    In addition, we call $(\Ocal, s)$ \emph{AZ-anti-ample} if $(\Ocal, s^{-1})$ is AZ-ample.
    \end{dfn}

    The following theorem shows the importance of AZ-(anti-)ampleness in noncommutative projective geometry.
    \begin{thm}[{\cite[Theorem 4.5]{artin1994noncommutative}}]
    \label{thm:AZ-characterization}
     Let $(\Ocal, s)$ be an algebraic pair for a $k$-linear category $\abcatname{C}$.
     \begin{enumerate}
    \item 
     We assume that 
     \begin{enumerate}[label=(AZ-\arabic*)]
        \item $\Ocal$ is noetherian,
        \item $\Hom_{\abcatname{C}}(\Ocal, \Ocal)$ is a right noetherian ring and $\Hom_{\abcatname{C}}(\Ocal,\Mcal)$ is a finite $\Hom_{\abcatname{C}}(\Ocal, \Ocal)$-module for all $\Mcal \in \abcatname{C}$,
        \item $s$ is AZ-ample. 
     \end{enumerate}
     Let $B:= B(\abcatname{C}, \Ocal, s)_{\geq 0}$.
     Then, $B$ is a right noetherian $\Nbb$-graded $k$-algebra and we have an isomorphism of algebraic triples
     \[
     (\abcatname{C}, \Ocal, s) \simeq (\qgr (B), \Ocal_B, (1)_B).
     \]
        
    \item Let $A$ be a right noetherian $\Nbb$-graded $k$-algebra satisfying the $\chi$-condition.
    Then, (AZ-1), (AZ-2) and (AZ-3) are satisfied for $(\qgr (A), \Ocal_A, (1)_A)$.
    \end{enumerate}
    \end{thm}

Now, we can prove Bondal-Orlov's reconstruction theorem in noncommutative projective geometry.
\begin{thm}
    \label{thm:re-main1}
    Let $A,B$ be noetherian locally finite $\Nbb$-graded $k$-algebras.
    We assume that $A,B$ have balanced dualizing complexes.
       We assume that $\qgr (A), \qgr (B)$ have canonical bimodules $\canmod_A, \canmod_B$, respectively.
    
        If $- \otimes \canmod_A, - \otimes \canmod_B$ are AZ-(anti-)ample, then 
        \[
         D^b(\qgr (A)) \simeq D^b(\qgr (B)) \Rightarrow \qgr (A) \simeq \qgr (B).
        \]
    
    \end{thm}

    \begin{proof}
        \label{proof:thm-main1}
        We assume that $- \otimes \canmod_A, - \otimes \canmod_B$ are AZ-anti-ample.
        Then, from \cref{prop:re-main1}, we have
        \[
           \anticanring(\qgr (A))_{\geq 0} \simeq \anticanring(\qgr (B) )_{\geq 0}
        \]
        as graded $k$-algebras.
        Since $A,B$ satisfies the $\chi$-condition, so
        \begin{align*}
           \qgr (A) \simeq \qgr (\anticanring(\qgr (A))_{\geq 0}) \simeq \qgr (\anticanring(\qgr (B))_{\geq 0}) \simeq \qgr (B)
        \end{align*}
        by \cref{thm:AZ-characterization}.
        In the case that $- \otimes \canmod_A, - \otimes \canmod_B$ are AZ-ample, we can prove the theorem in a similar way.
        
    \end{proof}

    In the following, we give an application of the argument of \cref{thm:re-main1} for AS-regular algebras.
    We recall the definitions of AS-regular algebras and AS-Gorenstein algebras below.

    \begin{dfn}[{\cite[Section 8]{artin1994noncommutative}, \cite[Definition 3.1]{minamoto2011structure}}]
        \label{def:AS-regular-Gorenstein}
        A locally finite $\Nbb$-graded $k$-algebra $A$ is called \emph{AS-regular} (resp. \emph{AS-Gorenstein}) over $R=A_0$ if $A$ satisfies the following conditions:
        \begin{enumerate}
            \item  $\gldim(A) = d < \infty$ and $\gldim(R) < \infty$
            
            \noindent(resp. $\injdim_A(A) = \injdim_{A^{\op}}(A) = d < \infty$),
            \item There exists an integer $l$ such that  
            \[
            \Extint^i_{A}(R, A) = \Extint^i_{A^{\op}}(R, A) =
            \begin{cases}
                R(l) & \text{if } i = d, \\
                0 & \text{otherwise},
            \end{cases}
            \]
            where $l$ is called the \emph{Gorenstein parameter} of $A$.
        \end{enumerate}
    \end{dfn}

    We also define a graded module twisted by a graded automorphism of a graded algebra.
    Let $A$ be an $\Nbb$-graded $k$-algebra.
    Let $\sigma$ be a graded automorphism of $A$.
    For a graded $A$-module $M$, we define $M_\sigma \in \Gr(A)$ by $M_\sigma := M$ as a graded $k$-module and the new right $A$-module structure is given by $m \cdot a = m\sigma(a)$ for $m \in M$ and $a \in A$.

    \begin{lem}[{\cite[Corollary 3.14]{minamoto2011structure}}]
        \label{lem:balanced dualizing of AS-regular}
        Let $A$ be a noetherian AS-regular algebra over $R=A_0$ of global dimension $n$ with the Gorenstein parameter $l$.
        Then, $A$ has a balanced dualizing complex given by $A_{\nu}(-l)[n] \in \Gr(A^{\en})$ for some graded automorphism $\nu$ of $A$.
        $\nu$ is called the generalized Nakayama automorphism of $A$.
    \end{lem}



    The application of \cref{thm:re-main1} is the following corollary.

    \begin{cor}
        \label{cor:re-main1}
        Let $A,B$ be noetherian AS-regular algebras over $R=A_0$ and $S=B_0$, respectively.
        Then, 
        \[
             D^b(\qgr (A)) \simeq D^b(\qgr (B)) \Rightarrow \qgr (A) \simeq \qgr (B).
            \]
        \end{cor}

    \begin{proof}
        \label{proof:cor-main1}
        If the Gorenstein parameters of $A,B$ are $l_A,l_B$ and the global dimensions of $A,B$ are $n_A,n_B$, then $A^{[l_A]}, B^{[l_B]}$ are AS-regular algebras of dimensions $n_A,n_B$ over $\left(A^{[l_A]}\right)_0, \left(B^{[l_B]}\right)_0$, respectively (\cite[Corollary 4.4]{mori2021acategorical}).
        In addition, the Gorenstein parameters of $A^{[l_A]}, B^{[l_B]}$ are $1$.
        From \cref{lem:balanced dualizing of AS-regular}, we have the balanced dualizing complexes of $A^{[l_A]}, B^{[l_B]}$ given by $A^{[l_A]}_{\nu_1}(-1)[n_A], B^{[l_B]}_{\nu_2}(-1)[n_B]$ for some graded automorphisms $\nu_1, \nu_2$ of $A^{[l_A]}, B^{[l_B]}$, respectively.
        This means that $- \otimes \canmod_{A^{[l_A]}}, - \otimes \canmod_{B^{[l_B]}}$ are AZ-anti-ample.
        Moreover, we have 
        \begin{align*}
            D^b\left(\qgr \left(A^{[l_A]}\right)\right) \simeq D^b(\qgr (A)) \simeq D^b(\qgr (B)) \simeq D^b\left(\qgr \left(B^{[l_B]}\right)\right)
        \end{align*}
        from \cref{lem:quasi-veronese}.
        So, we apply \cref{thm:re-main1} to obtain $\qgr \left(A^{[l_A]}\right) \simeq \qgr \left(B^{[l_B]}\right)$ and then $\qgr (A) \simeq \qgr (B)$.

    \end{proof}

\printbibliography 


\end{document}

%% file: mizuno_preamble.tex
\usepackage[driverfallback=dvipdfm]{hyperref}
\usepackage{amscd}
\usepackage{amsmath, amssymb}
\usepackage{mathrsfs}
\usepackage{multicol}
\usepackage{ifthen}
\usepackage{threeparttable}
\usepackage{caption}
\usepackage{bm}
\usepackage[all]{xy}
\usepackage{amsthm}
\usepackage[dvipsnames]{xcolor}
\usepackage{url}
\usepackage{paralist} 
\usepackage{enumerate}
\usepackage{enumitem}
\usepackage{latexsym}
\usepackage{arydshln}
\usepackage{multicol}
\usepackage{listings}
\usepackage{lineno}
\usepackage[capitalize]{cleveref}
\usepackage{tikz}
\usepackage{marginnote} 
\usepackage{comment}
\usetikzlibrary{cd}

\hypersetup{
    setpagesize=true, 
    bookmarksnumbered=false, 
    bookmarksopen=true, 
    colorlinks=true,
    linkcolor=VioletRed, 
    citecolor=SkyBlue,
    }

\newtheorem{thm}{Theorem}[section]

\newtheorem{cor}[thm]{Corollary}
\newtheorem{lem}[thm]{Lemma}
\newtheorem{sublem}[thm]{Sub-Lemma}
\newtheorem{prop}[thm]{Proposition}

\theoremstyle{definition}
\newtheorem{dfn}[thm]{Definition}
\newtheorem*{dfn*}{Definition}
\newtheorem{nota}[thm]{Notation}
\newtheorem{eg}[thm]{Example}

\theoremstyle{remark}
\newtheorem{rmk}[thm]{Remark}

\numberwithin{equation}{subsection}

\newcommand{\Nbb}{\mathbb{N}} 
\newcommand{\Zbb}{\mathbb{Z}} 

\newcommand{\Acal}{\mathcal{A}}
\newcommand{\Bcal}{\mathcal{B}}
\newcommand{\Ccal}{\mathcal{C}}

\newcommand{\Hcal}{\mathcal{H}}

\newcommand{\Mcal}{\mathcal{M}}
\newcommand{\Ncal}{\mathcal{N}}
\newcommand{\Ocal}{\mathcal{O}}

\newcommand{\Scal}{\mathcal{S}}
\newcommand{\Tcal}{\mathcal{T}}

\newcommand{\Lbf}{\mathbf{L}}

\newcommand{\Rbf}{\mathbf{R}}

\newcommand{\torfunct}[1]{\Gamma_{\mathfrak{m}_{#1}}}
\newcommand{\Rtorfunct}[1]{\R\Gamma_{\mathfrak{m}_{#1}}}

\newcommand{\dgcatname}[1]{\mathcal{#1}}
\newcommand{\dgfunname}[1]{#1}
\newcommand{\dgmodname}[1]{#1}

\newcommand{\abcatname}[1]{\mathcal{#1}}

\newcommand{\abobjname}[1]{\mathcal{#1}}

\newcommand{\qgrobjname}[1]{\mathcal{#1}}

\newcommand{\RHom}{\R\Hom}
\newcommand{\RHomint}{\R\underline{\Hom}}

\newcommand{\qV}{qV}

\newcommand{\R}{\Rbf}
\newcommand{\LD}{\Lbf}
\newcommand{\Q}{Q}
\newcommand{\U}{U}
\newcommand{\RQ}{\R Q}
\newcommand{\RprimeQ}{\R^{\prime}Q}

\newcommand{\Romega}{\R\omega}

\newcommand{\Rtau}{\R\tau}
\newcommand{\Rprimetau}{\R^{\prime}\tau}

\newcommand{\canring}{R_{\text{can}}}
\newcommand{\anticanring}{R_{\text{anti-can}}}
\newcommand{\canmod}{\mathcal{K}}

\newcommand{\Cdg}{C_{\dg}}
\newcommand{\Ddg}{D_{\dg}}

\DeclareMathOperator{\Hom}{Hom}
 
\DeclareMathOperator{\Ker}{Ker}
\DeclareMathOperator{\Cok}{Cok}
\DeclareMathOperator{\Ext}{Ext}
\DeclareMathOperator{\Cone}{Cone}

\DeclareMathOperator{\Homint}{\underline{Hom}}
\DeclareMathOperator{\Extint}{\underline{Ext}}

\DeclareMathOperator{\Proj}{Proj}

\DeclareMathOperator{\Coh}{Coh}
\DeclareMathOperator{\Perf}{Perf}

\DeclareMathOperator{\Mod}{Mod}
\DeclareMathOperator{\QMod}{QMod}
\DeclareMathOperator{\Gr}{Gr}
\DeclareMathOperator{\gr}{gr}
\DeclareMathOperator{\Tor}{Tor}
\DeclareMathOperator{\tor}{tor}
\DeclareMathOperator{\QGr}{QGr}
\DeclareMathOperator{\qgr}{qgr}
\DeclareMathOperator{\BiGr}{BiGr}

\DeclareMathOperator{\QBiGr}{QBiGr}

\DeclareMathOperator{\Locfree}{loc}
\DeclareMathOperator{\bgm}{bi}

\DeclareMathOperator{\serrefunct}{\Scal}

\DeclareMathOperator{\gldim}{gl.dim}
\DeclareMathOperator{\injdim}{inj.dim}

\DeclareMathOperator{\op}{op}
\DeclareMathOperator{\en}{en}
\DeclareMathOperator{\Ob}{Ob}

\DeclareMathOperator{\dg}{dg}
\DeclareMathOperator{\dgcat}{dgCat}

\DeclareMathOperator{\dgmod}{dgMod}
\DeclareMathOperator{\Acydg}{Acy_{dg}}
\DeclareMathOperator{\yoneda}{Y}
\DeclareMathOperator{\sfmod}{SF}
\DeclareMathOperator{\Perfdg}{Perf_{dg}}
\DeclareMathOperator{\hproj}{h-proj}
\DeclareMathOperator{\Fundg}{Fun_{dg}}

\DeclareMathOperator{\idfun}{Id}

